\documentclass{amsart}
\usepackage[utf8]{inputenc}
\usepackage{tikz,pgfkeys, pdfpages}
\usepackage{tikz-3dplot,forloop} 
\usetikzlibrary{arrows,calc,backgrounds,intersections} 
\usepackage{subcaption}

\usepackage[nointegrals]{wasysym}
\usepackage{enumitem}
\usepackage{amssymb, amsmath,stmaryrd}
\usepackage{amsthm,verbatim}
\usepackage{breqn}
\usepackage{mathtools}
\usepackage{graphicx}
\usepackage{color}
\usepackage{xcolor}
\usepackage{hyperref}
\usetikzlibrary{calc,fadings,decorations.pathreplacing}
\usetikzlibrary{positioning}
\usepackage{ marvosym }

\renewcommand{\div}{{\operatorname{div}}}

\newcommand{\bp}{\mathbf{p}}
\newcommand{\bq}{\mathbf{q}}
\newcommand{\bx}{\mathbf{x}}
\newcommand{\tri}[1]{\mathcal{T}_{#1}}

\newcommand{\domain}{{\Omega}}
\newcommand{\dd}{d}

\newcommand{\R}{\mathbb{R}^{\dd}}

\newcommand{\Vd}{{\bf H}^1_{\Gamma_D}(\domain)}

\newcommand{\qdelta}[1]{\mathbf{q}^\Delta_{#1}}

\newcommand{\bfsigma}{\boldsymbol{\sigma}}
\newcommand{\bftau}{\boldsymbol{\tau}}
\newcommand{\Hdiv}{H(\div;\Omega)}
\newcommand{\HdivN}{H_{\Gamma_D}(\div;\Omega)}
\newcommand{\HdivG}{H_{\Gamma_D,g}(\div;\Omega)}

\newcommand{\patch}[1]{\omega_{#1}}
\newcommand{\trito}[2]{\tri{#1}\rightarrow\tri{#2}}

\newcommand{\Rset}[1]{\mathcal R^{#1}_{\trito{\ell}{m}} }
\newcommand{\osc}{\textit{osc}}
\newcommand{\facets}{\mathcal E}
\newcommand{\patchsymbol}
{\mathrlap{{\small \hexstar}}{\hexagon}}
\newcommand{\vertiii}[1]{{\left\vert\kern-0.25ex\left\vert\kern-0.25ex\left\vert #1 
    \right\vert\kern-0.25ex\right\vert\kern-0.25ex\right\vert}}

\newcommand{\norm}[1]{\| #1 \|}
\newcommand{\bn}{{\mathbf n}}
\newcommand{\bff}{{\mathbf f}}
\newcommand{\bepsilon}{{\boldsymbol\varepsilon}}

\newcommand{\Vh}{V_h}
\newcommand{\cV}{\mathcal V}
\newcommand{\Vhexpl}{P^k(\mathcal T)}
\newcommand\data{\mathbf{D}}
\newcommand\AS{\mathbb{A}_s}
\newcommand\toterr[4]{\sigma(#1;#2,#3,#4)}

\newtheorem{Theorem}{Theorem}[section]
\newtheorem{Lemma}[Theorem]{Lemma}

\theoremstyle{remark}
\newtheorem{Remark}[Theorem]{Remark}

\numberwithin{equation}{section}

\newcommand\DB{\begin{color}{blue}}
\newcommand\BD{\end{color}}

\begin{document}

\title[The Prager--Synge theorem for AFEM]{The Prager--Synge theorem in reconstruction based a posteriori error estimation}

\author[F. Bertrand]{Fleurianne Bertrand}
\address{Institut f\"ur Mathematik, Humboldt
Universit\"at zu Berlin, Unter den Linden 6, 10099 Berlin, Germany}
\email{fb@math.hu-berlin.de}
\thanks{The first author gratefully acknowledges support by the German Research
Foundation (DFG) in the Priority Programme SPP 1748 \textit{Reliable simulation
techniques in solid mechanics} under grant number BE6511/1-1.}

\author[D. Boffi]{Daniele Boffi}
\address{Dipartimento di Matematica ``F. Casorati'', University of Pavia, Italy and Department of Mathematics and System Analysis, Aalto University, Finland}
\email{daniele.boffi@unipv.it}
\thanks{The second author is member of the INdAM Research group GNCS and his research is partially supported by IMATI/CNR and by PRIN/MIUR}

\subjclass[2010]{65N30, 65N50}

\date{June 2019}

\begin{abstract}

In this paper we review the hypercircle method of Prager and Synge. This theory inspired several studies and induced an active research in the area of a posteriori error analysis. In particular, we review the Braess--Sch\"oberl error estimator in the context of the Poisson problem. We discuss adaptive finite element schemes based on two variants of the estimator and we prove the convergence and optimality of the resulting algorithms.

\end{abstract}

\maketitle
\section{Introduction}

In this paper we review the hypercircle method introduced by Prager and Synge~\cite{PraSyn:47} and some of its consequences for the a posteriori analysis of partial differential equations.
We believe that it is useful to discuss a paper that has been the object of several studies and has induced an active research in the area of a posteriori analysis of partial differential equations.
On the one hand, it turns out that the hypercircle method is well appreciated by people working in the field, but less known by applied mathematicians with a less deep knowledge of a posteriori error analysis. On the other hand, we think that it is useful to discuss the consequences of the hypercircle method for a posteriori error analysis after some years of active research in the field, which has led in particular to a nowadays mature study of adaptive finite element schemes.
The hypercircle method provides a natural way to get guaranteed upper bounds for the error associated to Galerkin approximations; the corresponding lower bounds are more difficult to obtain and have been widely investigate in the literature.

It is now interesting to address the question whether an error estimator based on the hypercircle technique provides an optimally convergent method when combined with an adaptive strategy. This topic is less studied (see~\cite{MR2776915,MR2875241}) and we shall see that the answer to this question is not immediate.

The hypercircle method, originally developed for elasticity problems, can be used for several examples of PDEs.
Starting from the pioneer work of Ladev\`eze and Leguillon~\cite{LadLeg:83}, the Prager--Synge idea has led to several applications to the finite element approximation of elliptic problems~\cite{AinOde:93,MR1648383,MR2114285,MR2318795,MR2373174,MR2500243,MR2520373,MR2551163,MR2684353,MR2765501,MR2888308,MR2980729,MR3018142} and of problems in elasticity~\cite{Bra:13,MR2241823,MR2595045}. Other examples of applications include discontinuous Galerkin approximation of elliptic problems~\cite{MR3249368} or for convection-diffusion problems~\cite{MR2601287}; finite element approximation of convection-diffusion and reaction-diffusion problems has been studied in~\cite{MR2559737,MR3033032}. The Stokes problem and two phase fluid-flow have been considered in~\cite{MR2995179,MR3266956}. An intense activity is related to multiscale and mortar elements~\cite{MR3033022,MR3129761} as well as to porous media and porous elasticity~\cite{MR3621260,MR3622156,MR3761018}. Obstacle and contact problems have been studied in~\cite{MR2425501,MR2684731,MR2872024}. Further examples of applications include Maxwell's equations~\cite{MR3649425}, $hp$ finite elements~\cite{MR3556071}, and eigenvalue problems~\cite{MR3702871,MR3061473,BBS}. An interesting unified approach is provided in~\cite{ErnVoh:15} where the $p$-robustness of the error estimator is considered.

The hypercircle technique leads naturally to two methods: the so called \emph{gradient reconstruction} (related to the construction of the function $\nabla v$ of Figure~\ref{fig:hypercircle}) and the \emph{equilibrated flux} approach (related to the construction of the function $\bfsigma^*$ of Figure~\ref{fig:hypercircle}).

We develop our study starting from the case of the Laplace operator and we shall focus on the equilibrated flux approach. More precisely, we are going to discuss what is generally known as Braess--Sch\"oberl error estimator~\cite{BraSch:08}. For this estimator an a posteriori error analysis is well known which has been shown to be robust in the degree of the used polynomial~\cite{BraSch:08,BraPilSch:09}. We refer the interested reader in particular to the nice unified framework presented in~\cite{MR3335498} for more details on these results and for a complete survey of the use of equilibrated flux recovery in various applications.

The convergence analysis of the adaptive finite element method driven by non-residual error estimators has been performed in~\cite{MR2776915} and~\cite{MR2875241}. Both references start from the remark that it is not possible to expect in general a contraction property of the error and the estimator between two consecutive refinement levels. Since~\cite{MR2776915} is based on an assumption on the oscillations that might not be satisfied in our case (i.e., the oscillations are dominated by the estimator), in this paper we adopt the abstract setting of~\cite{MR2875241}. The Braess--Sch\"oberl estimator is considered in~\cite[Section~3.5]{MR2875241} where it is claimed that, up to oscillations, it is equivalent to the standard residual error estimator. We shall see that this property is not so immediate and that the consequence analysis has to be performed with particular care. In our paper we consider two variants of the Braess--Sch\"oberl estimator: the first one is the most standard and it is based on single elements (we denote it by $\eta^\Delta$); the second one is more elaborate and is based on patches of elements (denoted by $\eta^{\patchsymbol{}}$). The estimator $\eta^\Delta$ has been introduced in~\cite{BraSch:08}, while $\eta^\patchsymbol{}$ has been considered in~\cite{BraPilSch:09}. We are going to show that actually $\eta^{\patchsymbol{}}$ is equivalent, up to oscillations, to a residual estimator arranged on patches of elements (see Section~\ref{se:equiv}). We could not prove an analogous result for the estimator $\eta^\Delta$, which we analyze directly in Section~\ref{se:optimalfleur}. In both cases we have to pay attention to the appropriate definition and to the analysis of the oscillation terms. Oscillations are defined on patches of elements and the theory of~\cite{MR2875241} is modified accordingly. In turn, we present a clean theory where the convergence and the optimality of the adaptive schemes based on $\eta^\Delta$ and on $\eta^{\patchsymbol{}}$ is rigorously proved.

The structure of the paper is the following: in Section~\ref{se:PS} we recall the main results of the Prager--Synge hypercircle theory~\cite{PraSyn:47}, in Section~\ref{se:BS} we review the equilibrated flux reconstruction by Braess and Sch\"oberl~\cite{BraSch:08,BraPilSch:09}, in Section~\ref{se:equiv} we show the equivalence of the estimator $\eta^{\patchsymbol{}}$ with the standard residual one. We are then ready to recall in Section~\ref{se:star} the main ingredients of the theory of~\cite{MR2875241} and to apply it to the adaptive finite element method based on $\eta^{\patchsymbol{}}$. Finally, Section~\ref{se:optimalfleur} shows how to apply directly the theory of~\cite{MR2875241} to the estimator $\eta^\Delta$ without bounding it in terms of the standard residual estimator.

\section{The Prager--Synge theory and its application to error estimates}
\label{se:PS}

We start this section by reviewing the main aspects of the hypercircle theory introduced by Prager and Synge in~\cite{PraSyn:47}. The theory was developed for the mixed elasticity equation: the problem under consideration was to seek 
$\mathfrak{u} \in \Vd$ with 
\begin{align}
\label{elasticity}
\div \mathcal C \bepsilon( \mathfrak{u} ) = \bff,
\end{align}
where $C$ is the linear relationship between the stress and the strain.

In this paper we deal with the Poisson problem where a simplified version of the Prager--Synge theory can be applied.

Given a polytopal domain $\Omega$ in $\mathbb{R}^d$ and $f\in L^2(\Omega)$, our problem is to find $u\in H^1_0(\Omega)$ such that
\begin{equation}
\label{eq:variational}
(\nabla u,\nabla v)=(f,v)\quad\forall v\in H^1_0(\Omega).
\end{equation}
In this context, it is convenient to describe the Prager--Synge theory with the help of the mixed Laplacian equations. More precisely, let us consider the following problem: given $f\in L^2(\Omega)$, find $u\in L^2(\Omega)$ and $\bfsigma\in\Hdiv$ such that
\begin{equation}
\left\{
\aligned
&(\bfsigma,\bftau)+(\div\bftau,u)=0&&\forall\bftau\in\Hdiv\\
&(\div\bfsigma,v)=-(f,v)&&\forall v\in L^2(\Omega).
\endaligned
\right.
\label{eq:mixed}
\end{equation}

Problem~\eqref{eq:mixed} corresponds to the case of homogeneous boundary conditions on $u$. Clearly, more general boundary conditions can be considered. For the sake of completeness, we write down explicitly the general formulation associated to mixed boundary conditions $u=g_D$ on $\Gamma_D$ and $\bfsigma\cdot\bn=g_N$ on $\Gamma_N$, where $\partial\Omega$ is split in a Dirichlet part $\Gamma_D$ and in a Neumann part $\Gamma_N$. Let $\HdivN$ and $\HdivG$ denote the subspaces of vectorfields in $\Hdiv$ with normal component vanishing or equal to $g_N$, respectively, on $\Gamma_N$. Then the problem is: find $u\in L^2(\Omega)$ and $\bfsigma\in\HdivG$ such that
\begin{equation*}
\left\{
\aligned
&(\bfsigma,\bftau)+(\div\bftau,u)=\langle\bftau\cdot\bn,g_D\rangle|_{\Gamma_D}&&\forall\bftau\in\HdivN\\
&(\div\bfsigma,v)=-(f,v)&&\forall v\in L^2(\Omega),
\endaligned
\right.
\end{equation*}
where the brackets in the first equation represent the duality pairing between $H^{1/2}(\Gamma_D)$ and $H^{-1/2}(\Gamma_D)$ which, in the case of smooth functions, can be interpreted as
\[
\langle\bftau\cdot\bn,g_D\rangle|_{\Gamma_D}=\int_{\Gamma_D}g_D\bftau\cdot\bn\,ds.
\]
In this more general setting the analogue of~\eqref{eq:variational} reads: find $u\in H^1_{\Gamma_D,g}(\Omega)$ such that
\[
(\nabla u,\nabla v)=(f,v)-\langle g_N,v\rangle|_{\Gamma_N}\quad\forall v\in u\in H^1_{\Gamma_D}(\Omega),
\]
where $u\in H^1_{\Gamma_D}(\Omega)$ and $u\in H^1_{\Gamma_D,g}(\Omega)$ denote the subspace of $H^1(\Omega)$ with boundary conditions on $\Gamma_D$ vanishing or equal to $g_D$, respectively.

All the following theory could be stated in this general setting, but for the sake of readability we present it in the case when $\Gamma_N=\emptyset$ (so that $\Gamma_D=\partial\Omega)$ and $g_D=0$.

\medskip
\noindent\textbf{The equilibrium condition.} Let $\bfsigma^*$ be any function in $\Hdiv$ satisfying the equilibrium equation $\div\bfsigma^*=-f$; then it is easily seen that
\[
(\bfsigma,\bfsigma^*)=-(\div\bfsigma^*,u)=(f,u)=-(\div\bfsigma,u)=(\bfsigma,\bfsigma)
\]
from which the following orthogonality is obtained
\begin{equation}
(\bfsigma,\bfsigma-\bfsigma^*)=0.
\label{eq:1st}
\end{equation}

Equation \eqref{eq:1st} says that $\bfsigma$ lies on a hypersphere having $\bfsigma^*$ for diameter. The center of the sphere is denoted by $K$ in Figure~\ref{fig:hypercircle}.

\medskip
\noindent\textbf{Gradients of $\mathbf{H^1_0(\Omega)}$.} Let now $\bfsigma''$ be the gradient of any function $v$ in $H^1_0(\Omega)$
\[
\bfsigma''=\nabla v.
\]
It follows that
\[
(\bfsigma,\bfsigma'')=(\bfsigma,\nabla v)=-(\div\bfsigma,v)=(f,v)
\]
and that
\[
(\bfsigma^*,\bfsigma'')=(\bfsigma^*,\nabla v)=-(\div\bfsigma^*,v)=(f,v),
\]
which imply
\begin{equation}
(\bfsigma-\bfsigma^*,\bfsigma'')=0
\label{eq:2nd}
\end{equation}

The orthogonality stated in~\eqref{eq:2nd} can be expressed by saying that $\bfsigma$ and $\bfsigma^*$ lie on the same hyperplane orthogonal to $\bfsigma''$.

Putting together the orthogonalities of Equations~\eqref{eq:1st} and~\eqref{eq:2nd} leads to the conclusion that $\bfsigma$ and $\bfsigma^*$ lie on the \emph{\textbf{hypercircle}} $\Gamma$ given by the intersection of the hypersphere defined by~\eqref{eq:1st} and the hyperplane given by~\eqref{eq:2nd}. Moreover, let $\widehat{\bfsigma''}$ be the foot of $\bfsigma''$ on the hyperplane; since $\bfsigma-\bfsigma^*$ is orthogonal to $\bfsigma$ and $\widehat{\bfsigma''}$ is the orthogonal projection of $\bfsigma$ onto $\bfsigma''$, we have the following orthogonality
\[
(\bfsigma-\widehat{\bfsigma''},\bfsigma-\bfsigma^*) = 0,
\]
which implies that the segment connecting $\bfsigma^*$ to $\widehat{\bfsigma''}$ is a diameter of the hypercircle $\Gamma$. The center of this hypercircle is denoted by $C$ in Figure~\ref{fig:hypercircle}.

The conclusion of this construction, summarized in Figure~\ref{fig:hypercircle}, is an energy bound with constant one which we state in the following theorem.


\begin{figure}

\makeatletter
\tikzset{
  use path/.code={\pgfsyssoftpath@setcurrentpath{#1}}
}

\makeatother

\tdplotsetmaincoords{40}
	{110}
{
\begin{tikzpicture}[scale=1.2,tdplot_main_coords, 
  declare function={dicri(\t,\th,\ph,\R)=%
  sin(\th)*sin(\ph)*(2+\R*cos(\t)/3+2*\R*sin(\t)/3)-%
  sin(\th)*cos(\ph)*(-2 +2*\R*cos(\t)/3 + \R*sin(\t)/3)+%
  cos(\th)*(1+2*\R*cos(\t)/3-2*\R*sin(\t)/3);}] 
  \pgfmathsetmacro{\R}{5}%
  \path  coordinate (T) at (3,-3,3) 
   coordinate (I) at (1,-1,2) 
   coordinate (n) at (2,-2,1) 
   coordinate (u) at (1, 2, 2) 
   coordinate (v) at (2, 1, -2); 
   \path[tdplot_screen_coords,shift={(I)},use as bounding box] (-1.2*\R,-1.2*\R)rectangle (1.2*\R,1.2*\R);

  \definecolor{myball}{RGB}{255,247,151}
  \begin{scope}[tdplot_screen_coords, on background layer] 
   \fill[ball color=myball, opacity=1] (I) circle (\R); 
   \path[overlay,name path=dicri] plot[variable=\x,domain=0:360,samples=73] 
   ({\x*1pt},{dicri(\x,\tdplotmaintheta,\tdplotmainphi,4)}); 
   \path[overlay,name path=zero] (0,0) -- (360pt,0); 
   \path[name intersections={of=dicri and zero,total=\t}] 
   let \p1=(intersection-1),\p2=(intersection-2) in 
   \pgfextra{\xdef\tmin{\x1}\xdef\tmax{\x2}}; 
  \end{scope} 
  \pgfmathsetmacro{\SmallR}{4} 
  \draw[dashed] plot[variable=\t,domain=\tmin:\tmax,samples=50,smooth] 
   ({1+2+\SmallR*cos(\t)/3+2*\SmallR*sin(\t)/3}, 
   {-1-2 +2*\SmallR*cos(\t)/3+ \SmallR*sin(\t)/3}, 
   {2+1+2*\SmallR*cos(\t)/3 - 2*\SmallR*sin(\t)/3 }); 
  \draw[thick,save path=\pathA] plot[variable=\t,domain=\tmax:\tmin+360,samples=50,smooth] 
   ({1+2+\SmallR*cos(\t)/3+2*\SmallR*sin(\t)/3}, 
   {-1-2 +2*\SmallR*cos(\t)/3+ \SmallR*sin(\t)/3}, 
   {2+1+2*\SmallR*cos(\t)/3 - 2*\SmallR*sin(\t)/3 }); 
  \path ({1+2+\SmallR*cos(\tmin)/3+2*\SmallR*sin(\tmin)/3}, 
   {-1-2 +2*\SmallR*cos(\tmin)/3+ \SmallR*sin(\tmin)/3}, 
   {2+1+2*\SmallR*cos(\tmin)/3 - 2*\SmallR*sin(\tmin)/3 }) coordinate (pmin)
   ({1+2+\SmallR*cos(\tmax)/3+2*\SmallR*sin(\tmax)/3}, 
   {-1-2 +2*\SmallR*cos(\tmax)/3+ \SmallR*sin(\tmax)/3}, 
   {2+1+2*\SmallR*cos(\tmax)/3 - 2*\SmallR*sin(\tmax)/3 }) coordinate (pmax);
  \begin{scope}[tdplot_screen_coords]
   \clip[shift={(I)}] (-1.2*\R,-1.2*\R)rectangle (1.2*\R,1.2*\R);   
   \path[fill=blue,fill opacity=0.2,even odd rule] let \p1=($(pmin)-(I)$),\p2=($(pmax)-(I)$),
   \p3=($(pmax)-(pmin)$),\n1={atan2(\y1,\x1)},\n2={atan2(\y2,\x2)},
   \n3={atan2(\y3,\x3)}
    in [use path=\pathA]  (pmin) arc(\n1:\n2-360:\R) 
    (1.5,-6) -- ++(\n3:{12cm/sin(\n3)}) -- ++(\n3+90:{12cm/sin(\n3)})
    -- ++(\n3+180:{12cm/sin(\n3)}) -- cycle;
  \foreach \v/\position in {T/above,I/below} { 
   \draw[fill=black] (\v) circle (0.7pt) node [\position=0.2mm] {$\ $}; 
  } 
  \end{scope}

\node[inner sep=0pt] (s) at ( 1.8, -2.322440841,   6.755118318){};

\node[inner sep=0pt] (gvs) at(  4.200000000,-3.67755915, -0.755118318){};

\node[inner sep=0pt] (u) at ( 5., -0.73350084e-1,   4.853299832){};

\draw[line width=0.5mm,->,orange] (u)--(s);
\draw[line width=0.5mm,->,orange] (u)--(gvs);
\draw[line width=0.5mm,->,orange] (s)--(gvs);

\node[inner sep=0pt] (O) at ( 0.2, 0.322440841, -2.755118318){};

\node[inner sep=0pt] (gv) at (6, -5.477559159, 0.144881682){};


\path [line width=0.5mm,->,red] (O) --  (gv);

\node[inner sep=0pt] (gv2) at (6, -5.177559159, -0.0544881682){\color{red}$\quad\quad\quad \bfsigma''=\nabla v  $};

\path [line width=0.5mm,->,blue,dashed] (O) -- node [below] {\color{blue} $ \nabla v' = \widehat   {\bfsigma''} \quad\quad\quad\quad\ $} (gvs);

\path [line width=0.5mm,->,blue,dashed] (O) -- node [right] {\color{black} K} (s);

\node[inner sep=0pt] (stext) at
( 1.8, -1.122440841,   4.755118318)
{\color{blue} $\ \bfsigma^*$};

\filldraw[red,->,line width=0.5mm] (O)--(gv)
;
\filldraw[red,->,line width=0.5mm] (gv)--(s)
;

\draw[line width=0.5mm,->,blue,dashed] (O)--(gvs);
\draw[line width=0.5mm,->,blue,dashed] (O)--(s);

\draw[line width=0.5mm,->,violet] (O)--(u);
\draw[line width=0.5mm,-,red] (gv)--(u);

\node[inner sep=0pt,rotate=105] (stext) at
( 1.8, -1.722440841,   1.755118318)
{$\color{brown} \|\bfsigma-\bfsigma^* \|  $};

\node[inner sep=0pt,rotate=73] (stext) at
( 1.8, -3.12440841,   1.755118318)
{$\color{brown} \|\nabla v'-\bfsigma^* \|  $};

\node[inner sep=0pt,rotate=55] (stext) at
( 1.8, -4.62440841,   1.755118318)
{$\color{red} \|\nabla v-\bfsigma^* \|  $};

\node[inner sep=0pt,rotate=31] (stext) at
( 1.8, -2.722440841,   -1.655118318)
{$\color{brown} \|\bfsigma-\nabla v' \|  $};

\node[inner sep=0pt,rotate=23] (stext) at
( 1.8, -2.722440841,   -0.305118318)
{$\color{red} \|\nabla u -\nabla v \|  $};

\node[inner sep=0pt] (stext) at
 ( 5., -0.73350084e-1,   4.053299832)
{$\color{violet} \ \ \ \ \quad \quad \quad \bfsigma = \nabla u  $};

\node[] (Ctext) at
( 1.8, -3.52440841,   1.755118318)
{$C$};

\node[] (Dtext) at
( 0.5, -4,   5.655118318)
{$\Gamma$};

\path [line width=0.5cm] (O) -- node [above left]
{ } (u);

\filldraw [blue] (O) circle(1pt);
\filldraw [orange] ( 0.2, 0.322440841, -2.755118318) circle(1pt);

\filldraw [red] ( 6, -5.477559159, 0.144881682) circle(1pt);
 
\draw[fill=black] (T) circle (1pt) node  {};
\draw[fill=black] (I) circle (1pt) node {};
\end{tikzpicture} }

	
    \caption{The hypercircle construction}
    \label{fig:hypercircle}
\end{figure}

\begin{Theorem}
Let $\bfsigma$ be the second component of the solution to~\eqref{eq:mixed}; let $\bfsigma^*$ be any function in $\Hdiv$ which satisfies the equilibrium condition $\bfsigma^*=-f$ in $\Omega$ and let $\bfsigma''$ be the gradient of any function in $H^1_0(\Omega)$. Then
\[
\|\widehat\bfsigma''\|\le\|\bfsigma\|\le\|\bfsigma^*\|,
\]
where $\widehat\bfsigma''$ is the multiple of $\bfsigma''$ lying in the hyperplane orthogonal to $\bfsigma''$ and containing $\bfsigma$ (see Figure~\ref{fig:hypercircle}).
\end{Theorem}

We now state another important consequence of the previous geometrical construction which applies to problem~\eqref{eq:variational} and which is usually referred to as Prager--Synge theorem.

\begin{Theorem}
Let $u$ be the solution of problem~\eqref{eq:variational}. Then it holds
\begin{align}
\norm{\nabla u -\nabla v}^2 +\norm{\nabla u -\bfsigma^*}^2 
=
\norm{\nabla v -\bfsigma^*}^2
\end{align}
for all $v\in H^1_0(\Omega)$ and all $\bfsigma^*\in\Hdiv$ satisfying the equilibrium condition $\div\bfsigma^*=-f$.
\label{th:PS}
\end{Theorem}

\begin{proof}
From the orthogonalities defining the hypersphere and the hyperplane $(\bfsigma,\bfsigma-\bfsigma^*)=(\bfsigma-\bfsigma^*,\bfsigma'')=0$ if follows immediately $(\bfsigma-\bfsigma'',\bfsigma-\bfsigma^*)=0$ which gives the results with the identifications $\nabla u=\bfsigma$ and $\nabla v=\bfsigma''$.
\end{proof}

%

The Prager-Synge theorem has been used in order to obtain error estimates in various contexts, starting from ~\cite{LadLeg:83}. We describe the application of Theorem~\ref{th:PS} in the case of the conforming finite element approximation of problem~\eqref{eq:variational}. Let $\Vh$ be a finite dimensional subspace of $H^1_0(\Omega)$ and consider the discrete problem: find $u_h\in\Vh$ such that
\begin{equation}
(\nabla u_h,\nabla v_h)=(f,v_h)\quad\forall v\in\Vh.
\label{eq:variational_h}
\end{equation}
We are going to consider a standard conforming $\Vh$, so that $u_h\in\mathcal P^k(\tri{})$, the space of continuous piecewise polynomials of degree less than or equal to $k$.

A direct application of Theorem~\ref{th:PS} with $v=u_h$ and $\bfsigma^*=\bq$ gives
\begin{equation}
\label{eq:appPS}
|\nabla u-\nabla u_h|\le\|\bq-\nabla u_h\|,
\end{equation}
where $\bq$ is any function in $\Hdiv$ with $\div\bq=-f$ in $\Omega$.
It turns out that the right hand side in~\eqref{eq:appPS} is a \emph{reliable} error estimator \emph{with constant one}. Clearly, this fundamental idea leads to a viable approach only if it is possible to construct $\bq$ in a practical way. This is what is generally called equilibrated flux reconstruction.
\begin{Remark}
In the case when $f$ is piecewise polynomial, a possible (not practical) definition of $\bq$ could be obtained by solving an approximation of the mixed problem~\eqref{eq:mixed}, so that $\bq$ is a discretization of $\bfsigma$. If $f$ is a generic function, a standard oscillation term will show up. A smart modification of this intuition is behind the Braess-Sch\"oberl construction presented later in this paper.
\end{Remark}

Ainsworth and Oden in~\cite[Chap.~6.4]{AinOde:00} show that $\bq$ can be efficiently constructed by solving local problems. Let $\mathcal E_{I}$ be the set of the interior edges of a shape-regular triangulation $\tri{}$. We will also denote by $\mathcal E_{B}$ the set of the boundary edges. In the case when problem~\eqref{eq:variational_h} is solved with polynomials of degree $k$, the reconstruction proposed in~\cite{AinOde:00} seeks $\qdelta{} = \bq - \nabla u_h$ such that
\begin{subequations}
\label{eq:fluxrecocond}
\begin{align}
\label{eq:fluxrecoconda}
&\div{ \qdelta{}} = - \Pi^k f - \div\ \nabla u_h\\
&\llbracket \qdelta{} \cdot \bn \rrbracket_E =
 -\llbracket
 \nabla u_h \cdot \bn
 \rrbracket_E \quad
 \quad\forall E \in \mathcal E_I,
\label{eq:fluxrecocondb}
\end{align}
\end{subequations}
where $\Pi^k f$ denotes the $L^2$ projection of $f$ onto polynomials of degree $k$.

\section{The Braess--Sch\"oberl construction}
\label{se:BS}

In~\cite{BraSch:08} Braess and Sch\"oberl show how to realize the above conditions \eqref{eq:fluxrecoconda} and \eqref{eq:fluxrecocondb} by exploiting some basic properties of the Raviart--Thomas finite element spaces. The resulting estimator is commonly called the Braess--Sch\"oberl error estimator.

The local problems can be solved on patches around vertices of the mesh. The construction has been extended to different problems and geometrical configurations, thus allowing for a very powerful and general equilibration procedure.
The reconstruction aims at defining $\qdelta{}$ in the broken Raviart--Thomas space of order $k$, that is
\begin{align}
RT^\Delta (\tri{}) = \{ \bq \in RT^k(T) \textrm{ for all } T \in \tri{} \}, \end{align}
where the Raviart--Thomas element is given by
\begin{align}
RT^k(T)=\{  \bp \in \mathbb{P}^{k+1}(T) \ : \ \bp(\bx) = \hat{\bp} (\bx) + \bx \tilde{p}, \ \hat{\bp}\in (\mathbb{P}^k(T))^d  \text{, } \tilde{p} \in \mathbb{P}^k(T) \}
\end{align}
and $\mathbb{P}^k(K)$ denotes the space of polynomials of degree at most $k$ on the domain $K$.
Clearly, since $u_h\in \mathcal P^k(\tri{})$, we will have that $\bq=\qdelta{}-\nabla u_h$ belongs to $RT^k(\tri{}) := RT^\Delta (\tri{}) \cap H (\div,\Omega) $ by virtue of the jump conditions~\eqref{eq:fluxrecocondb}.

The Braess-Sch\"oberl reconstruction is performed as follows.
Let $\mathcal V$ denote the set of vertices of the triangulation, $\nu\in \mathcal V$ a vertex, and $\patch{\nu}$ the patch of elements sharing the vertex $\nu$
\begin{equation}
  \omega_\nu := \bigcup \{ T \in \tri{} : \nu \mbox{ is a vertex of } T \}.
  \label{eq:vertex_patch}
\end{equation}
Let $\phi_\nu$ be the continuous piecewise linear Lagrange function with $\phi_\nu(\nu) =1$ and whose support is $\patch{\nu}$, (that is, the hat function equal to one at the node $\nu$), so that the following partition of unity property holds
\begin{equation}
  1 \equiv \sum_{\nu \in \cV} \phi_\nu \mbox{ on } \Omega.
  \label{eq:partition_of_unity}
\end{equation}
Hence $\qdelta{}$ can be decomposed into functions living on vertex patches, i.e.
$$\qdelta{} = \sum_{\nu \in \cV
} \phi_\nu \qdelta{}= \sum_{\nu \in \cV
} \qdelta{\nu},$$
where $\text{supp}(\qdelta{\nu}) =\patch{\nu}$ and $\qdelta{\nu} \cdot \bn = 0 \text{ on } \partial \omega_\nu$.

Since each facet belongs to two elements the conditions \eqref{eq:fluxrecoconda} and \eqref{eq:fluxrecocondb} mean that the function $\qdelta{\nu}$ has to fulfill

\begin{align}
\label{eq:local_equilibration}
\begin{split}
\begin{cases}
\div\ \qdelta{\nu} = 
-((f+\Delta u_h),\phi_\nu)_T
&\quad\text{  in each }T \in \patch{\nu}\\
[[\qdelta{\nu} \cdot\bn ]]=
-( [[\nabla u_h \cdot {\bn}]], \phi_\nu)_E
&\quad\text{ on each interior edge $E$ of }\patch{\nu}\\
\qdelta{\nu} \cdot \bn = 0
&\quad\text{  on }\partial\patch{\nu}.
\end{cases}
\end{split}
\end{align}

It is common to use a notation where the dependence on the discrete solution $u_h$ is made explicit, so that in general we are going to denote the reconstruction by $\qdelta{}(u_h)$ or its contribution coming from a patch $\qdelta{\nu}(u_h)$.

Two options are now given for the design of an error indicator based on the above reconstruction.
The first one, introduced in~\cite{BraSch:08}, considers directly the quantity $\qdelta{}(u_h)$ on each single element
\begin{align}
\eta^{\Delta}_T(u_h)     = 
\|\qdelta{}(u_h) \|_{0, T}
\qquad
\eta^{\Delta}(u_h,\tri{})     = 
\left( \sum\limits_{T\in \mathcal T}
(\eta^{\Delta}_T(u_h))^2 \right)^{1/2},
\end{align}
while the second on, presented in~\cite{BraPilSch:09}, is based on patches of elements
\begin{align}
\eta^{\patchsymbol}_\nu(u_h)   = 
\|\qdelta{\nu}(u_h) \|_{0, \patch{\nu}}
\qquad
\eta^{\patchsymbol}
(u_h,\tri{})     = 
\left( \sum\limits_{T
\in  \tri{}}
\sum\limits_{\nu
\in  \mathcal V_T}
(\eta^{\patchsymbol}_\nu(u_h))^2 \right)^{1/2}.
\end{align}

The estimators $\eta^{\Delta}_T(u_h)$ and $\eta^{\patchsymbol}_\nu(u_h)$ are clearly not equivalent.
People usually tend to consider $\eta^\Delta$ as the standard Breass--Sch\"oberl estimator, but it is clear that for the analysis sometimes $\eta^\patchsymbol{}$ may be more convenient.

An a posteriori analysis for both estimators is available in the sense that both satisfy a global reliability
\begin{equation}
    \|| u- u_\ell \||^2 \leq \eta^2({u_l,\tri{}})+\osc_{\tri{}}^2 (f)
\label{eq:BSrel}
\end{equation}
and a global efficiency
\begin{equation}
    \eta^2({u_\ell,\tri{}}) \leq  \|| u- u_\ell \||^2+\osc_{\tri{}}^2(f)
\label{eq:BSeff}
\end{equation}
up to oscillations (see, in particular, \cite[Theorems~9.4 and~9.5]{Bra:13}, and~\cite{BraSch:08,BraPilSch:09}). The definition of the oscillation terms need particular attention. We shall comment on that in the next sections.

Explicit formulas in the case $d = 2$ for the computation of $\qdelta{\nu}$ are given in \cite{bertrandCISM}. The direct construction is extended to $d=3$ in \cite{CaiZha:12b}. 

\section{Equivalence with the residual error estimator}
\label{se:equiv}

In this section we are going to show that, up to an oscillation term, the estimator $\eta^{\patchsymbol{}}$ is equivalent to an estimator based on the standard residual error estimator.

A crucial step for the analysis of the convergence of the adaptive scheme based on the Braess--Sch\"oberl error estimator is its local equivalence with a standard residual error estimator. This fact has been observed (without rigorous proof) in~\cite[Section~3.5]{MR2875241} and it has been used (without oscillations) in~\cite[Equation~2.17]{MR2776915}. The interested reader is referred to~\cite[Section~8]{CarFeiPagPra:14} for a more elaborate discussion about the equivalence between residual and non-residual error estimators.

The standard residual estimator for Laplace equation is based on two contributions: the element and jump residuals
\begin{equation}
\aligned
&R_T(v) = ( f + \Delta v )|_T   \\
&J_E ( v) = ( [[ \nabla v ]] \cdot \bn )|_E,
\endaligned
\end{equation}
where $T$ is an element of the triangulation $\tri{}$ and $E$ is a facet in the set of facets $\facets$. 
The residual estimator for $T\in\tri{}$ then reads
\begin{align}
\eta_{res}^2 (u_h,T) = \|h_T R_T(u_h)\|_{0,T}^2 +\|h_T^{1/2} J_{\partial T}(u_h)\|_{0,\partial T}^2,
\end{align}
where $J_{\partial T}(u_h)$ is viewed as a piecewise function over $\partial T$ and where as usual $h_T$ denotes the diameter of the element $T$.

It is well known that the error estimator defines a functional $R(u_h) \in (H_0^1(\Omega))^\prime$ as follows
\begin{equation}
\aligned
\langle R(u_h),v \rangle &= 
\sum\limits_{T\in\tri{}} (R_T(u_h), v)_T
+
\sum\limits_{E\in\mathcal E} 
\langle
J_E(u_h), v
\rangle_E
\\&= (f, v)_{\Omega}-(\nabla u_h, \nabla v)_{\Omega}
=(\nabla (u-u_h), \nabla v)_{\Omega}&&\forall v\in H_0^1(\Omega).
 \endaligned
 \end{equation}

The global residual error estimator on a triangulation $\mathcal T$ is usually defined by adding up the local contributions
\begin{align}
\eta_{res} (u_h,\mathcal T) = 
\left(\sum\limits_{T\in \mathcal T}
\eta_{res}^2 (u_h,T) 
\right)^{1/2}.
\end{align}
Unfortunately, no equivalence holds in general between $\eta^\patchsymbol{} (u_h, T)$ and $\eta_{res} (u_h, T)$; a crucial difference between the two estimators is that if an element $T$ belonging to the patch $\patch{\nu}$ is refined and the discrete solution $u_h$ doesn't change, then the error is not reduced, but the estimator $\eta_{res} (u_h,T)$ decreases because of the reduction of the mesh-size; on the other hand, $\eta^\patchsymbol{} (u_h,\patch{\nu})$ may not decrease since it is based on the equilibration procedure that might generate a reconstruction that is not different from the one computed on the coarser mesh.

An interesting alternative, described in~\cite{BraPilSch:09} for piecewise constant $f$, consists in building a residual error estimator which is based on element patches, so that the comparison with $\eta^\patchsymbol{}$ is more natural. This leads, for every node $\nu$ with corresponding Lagrangian function $\phi_\nu$, to the following definition
\begin{equation}
\aligned
&R_{\nu,T}(v) = \phi_{\nu}( f + \Delta v )|_T   \\
&J_{\nu,E} ( v) = \phi_{\nu}( [[ \nabla v ]] \cdot \bn )_E.
\endaligned
\end{equation}
We denote the corresponding global estimator by
\begin{align}
\tilde \eta_{res} (u_h,\tri{}) = 
\left(
\sum\limits_{T
\in  \tri{}}
\sum\limits_{\nu
\in  \mathcal V_T}
\tilde\eta_{res} (u_h,\nu) 
\right)^{1/2},
\end{align}
with 
\begin{align}
\tilde \eta^{res}_\nu (u_h) = 
\eta^{res} (u_h,\omega_\nu).
\end{align}

The next lemma states the local equivalence between $\eta^\patchsymbol{}$ and the patchwise residual estimator $\tilde\eta^{res}$.

\begin{Lemma}
Let $u_h$ be the solution of the variational formulation \eqref{eq:variational_h} and consider a node $\nu$ of the triangulation $\tri{}$. Then, it holds
\label{lem:reseqrec}
\begin{align}
    \eta^{\patchsymbol}_\nu(u_h) \simeq \tilde \eta^{res}_\nu (u_h)
\end{align}
up to the oscillation term $\sum\limits_{T \in \omega_\nu }\| h_T(id- \Pi^{k-1}_T)(f) \|_T$, that is
\begin{subequations}
\begin{align}
\label{eq:lesssim}
&\| \qdelta{\nu}(u_h) \|_{0,\omega_\nu} \lesssim \tilde \eta^{res}_\nu (u_h)+\sum\limits_{T \in \omega_\nu }\| h_T(id- \Pi^{k-1}_T)(f) \|_T \\
\label{eq:gtrsim}
&\| \qdelta{\nu}(u_h) \|_{0,\omega_\nu} +\sum\limits_{T \in \omega_\nu }\| h_T(id- \Pi^{k-1}_T)(f) \|_T \gtrsim \tilde \eta^{res}_\nu (u_h).
\end{align}
\end{subequations}

\end{Lemma}

\begin{proof}
Let us start with the upper bound~\eqref{eq:lesssim}. When $f$ is piecewise polynomial of degree ${k-1}$, from \cite[Theorem~7]{BraPilSch:09} we have
\begin{align}
\label{Theo7BraPilSch}
\| \qdelta{\nu}(u_h) \|_{0,\omega_\nu} \lesssim 
\sup\limits_{
\underset{v \in H^1(\patch{\nu})}{\| v \|_1 = 1}
} 
\sum\limits_{T\in \patch{\nu}} (R_{\nu,T}(u_h),v)_{0,T}
+\sum\limits_{E\in \mathcal E_I(\patch{\nu})} (J_{\nu,E} (u_h),v)_{0,E}
\end{align}
and thus, using standard scaling arguments, 
\begin{align*}
\| \qdelta{\nu}(u_h) \|_{0,\patch{\nu}} &\lesssim 
\sup\limits_{
\underset{v \in H^1(\patch{\nu})}{\| v \|_1 = 1}
} 
\sum\limits_{T\in \patch{\nu}} 
 \| R_{T}(u_h)\|_{0,T} \|v\|_{0,T}
+ 
\sum\limits_{E \in \mathcal E^I_\nu}
\|J_{E}(u_h)\|_{0,E} \|v\|_{0,E} \\
&\lesssim 
\sup\limits_{
\underset{v \in H^1(\patch{\nu})}{\| v \|_1 = 1}
} 
\sum\limits_{T\in \patch{\nu}}  h_T\| R_{T}(u_h)\|_{0,T} \|v\|_{1,\omega_\nu}
+ \sum\limits_{E\in \mathcal E^I_\nu} h_E^{1/2}\|J_{E}(u_h)\|_{0,E}
\|v\|_{1,\omega_\nu}\\
&\lesssim 
\sum\limits_{T\in \patch{\nu}}  h_T\| R_{T}(u_h)\|_{0,T}
+\sum\limits_{E\in \mathcal E^I_\nu} h_E^{1/2}\|J_{E}(u_h)\|_{0,E}.
\end{align*}
If now $f$ is a generic function in $L^2(\Omega)$, then the first term in \eqref{Theo7BraPilSch} transforms into 
\[
\sup\limits_{
\underset{v \in H^1(\patch{\nu})}{\| v \|_1 = 1}
} 
\sum\limits_{T\in \patch{\nu}} (\Pi^{k}_T \left(\phi_{\nu}( f + \Delta u_h ) \right)|_T,v)_{0,T}
+((id-\Pi^{k}_T) \left(\phi_{\nu}( f + \Delta u_h ) \right)|_T,v)_{0,T}
\]
so that it remains to show that
\[
\sup\limits_{
\underset{v \in H^1(\patch{\nu})}{\| v \|_1 = 1}
} 
\sum\limits_{T\in \patch{\nu}} ((id-\Pi^{k}) \left(\phi_{\nu} f  \right),v)_{0,T}
\lesssim \sum\limits_{T \in \omega_\nu }\| h_T(id- \Pi^{k-1})(f) \|_T.
\]
Indeed
\begin{align*}
\sup\limits_{
\underset{v \in H^1_T(\patch{\nu})}{\| v \|_1 = 1}
} 
&\sum\limits_{T\in \patch{\nu}} ((id-\Pi^{k}_T) \left(\phi_{\nu} f  \right),v)_{0,T}\\
&= 
\sup\limits_{
\underset{v \in H^1(\patch{\nu})}{\| v \|_1 = 1}
} 
\sum\limits_{T\in \patch{\nu}} ((id-\Pi^{k}_T) \left(\phi_{\nu}f\right)|_T,v-(v,1)_T)_{0,T}\\
&= 
\sum\limits_{T\in \patch{\nu}} \| h_T (id-\Pi^{k}_T) \left(\phi_{\nu}f\right)\|_{0,T}\\
\end{align*}
and 
\begin{align*}
\| (id-\Pi^{k}_T) \left(\phi_{\nu}f\right)\|_{0,T} ^2&= 
((id-\Pi^{k}_T) \left(\phi_{\nu}f\right),\phi_{\nu}f)\\
&=((id-\Pi^{k}_T) \left(\phi_{\nu}f\right),\phi_{\nu}f-\phi_\nu\Pi_T^{k-1}f)\\
&\leq \| (id-\Pi^{k}_T) \left(\phi_{\nu}f\right) \|_{0,T} \| \phi_{\nu}f-\phi_\nu\Pi^{k-1}_T f\|_{0,T}\\
&\leq \| (id-\Pi^{k}_T) \left(\phi_{\nu}f\right) \|_{0,T}  \| f-\Pi^{k-1}_Tf\|_{0,T}
\end{align*}
since $\max\limits_{\bx\in T}(\phi_\nu(\bx))=1$. This implies
\begin{align*}
\| (id-\Pi^{k}_T) \left(\phi_{\nu}f\right)\|_{0,T} \leq  \| f-\Pi^{k-1}_Tf\|_{0,T}.
\end{align*}

Let us now show how to prove the lower bound~\eqref{eq:gtrsim}.
Recall that \eqref{eq:local_equilibration} implies
\begin{align*}
\label{eq:equilibration}
(\qdelta{\nu}, \nabla v) &= 
-\sum\limits_{T\in \omega_\nu}
((\Pi^{k-1}_Tf+\Delta u_h)\phi_\nu,v)_T
-\sum\limits_{E\in \mathcal E^I_\nu}
\langle [[\nabla u_h \cdot {\bn}]]_E \phi_\nu,v
\rangle_E \\
& = -\sum\limits_{T\in \omega_\nu}
(R_{\nu,T}(u_h),v)_T + ((f-\Pi^{k-1}_Tf)\phi_\nu,v)
-\sum\limits_{E\in \mathcal E^I_\nu}
\langle J_{\nu,E}(u_h),v
\rangle_E
\end{align*}
for any $v\in H^1(\patch{\nu})$ satisfying either zero boundary conditions or $(v,1)_{\patch{\nu}}=0$ in the case when $\nu$ is an internal node.
Now, take
\[
v=\tilde{v}-\int_{\patch{\nu}}\tilde{v},
\]
with $\tilde{v}$ is defined as follows
\begin{align*}
\tilde{v} = 
\sum\limits_{T \in \omega_\nu}
\phi^3_{T}+
\sum\limits_{E \in \mathcal E^I_{\nu}}
\phi^{k+2}_{E}-\Pi^{k+1}_T(\phi^{k+2}_{E}),
\end{align*}
where  $\phi^3_{T}$ denotes the cubic Lagrange bubble function corresponding to the barycenter of $T$ and $\phi^{k+2}_{E}$ one of the Lagrange functions of degree $k+2$ associated to the edge $E$. Since the norm of $v$ is bounded, we have
\[
\aligned
\sum\limits_{T \in \omega_\nu }\| h_T(id- \Pi^{k-1}_T)(f) \|_T  + \| \qdelta{\nu} \|
&\gtrsim\sum\limits_{T \in \omega_\nu }\| h_T(id-\Pi^{k-1}_T)(f) \|_T+\Big\vert(\qdelta{\nu},\nabla v)\Big\vert\\
&=\sum\limits_{T \in \omega_\nu }\| h_T(id-\Pi^{k-1}_T)(f) \|_T+\Big\vert(\qdelta{\nu},\nabla\tilde{v})\Big\vert.
\endaligned
\]
Moreover, we have that
\[
\aligned
\sum\limits_{T \in \omega_\nu }\| h_T(id- \Pi^{k-1}_T)(f) \|_T&=
\sup_{v\in L^2(\patch{\nu}}\sum\limits_{T \in \omega_\nu }\frac{(h_T(id- \Pi^{k-1}_T)(f),v)_T}{\|v\|_0}\\
&\ge \sum\limits_{T \in \omega_\nu }\frac{(h_T(id- \Pi^{k-1}_T)(f),\phi^3_T)_T}{\|\phi^3_T\|_0}\\
&\gtrsim\sum\limits_{T \in \omega_\nu }((id- \Pi^{k-1}_T)(f),\phi^3_T)\\
&\gtrsim\sum\limits_{T \in \omega_\nu }((id- \Pi^{k-1}_T)(f),\phi^3_T\phi_\nu).
\endaligned
\]
By inserting the expression for $\tilde{v}$ and by evaluating the different terms separately we finally obtain
\begin{align*}
\sum\limits_{T\in\omega_\nu }&\| h_T(id-\Pi^{k-1}_T)(f) \|_T 
+ \Big\vert(\qdelta{\nu}, \nabla (\phi^3_{T}+\phi^{k+2}_{E}-\Pi^{k+1}(\phi^{k+2}_{E})))\Big\vert \\
&\gtrsim
\Big\vert- \sum\limits_{T \in \omega_\nu}
((f+\Delta u_h)\phi_\nu,\phi^3_{T})_T
- \sum\limits_{E \in \mathcal E^I_{\nu}}
\langle [[\nabla u_h \cdot {\bn}]]_E \phi_\nu,
\phi^{k+2}_{E}-\Pi_{k+1}(\phi^{k+2}_{E})
\rangle_E\Big\vert \\ &
\gtrsim  \sum\limits_{T \in \omega_\nu} h_T \|f +\Delta u_h \|_{0,T} +
\sum\limits_{E \in \mathcal E^I_{\nu}} h_T^{1/2} \|J_{E}(u_h)\|_{0,E}.
\end{align*}
\end{proof}

\section{Optimal convergence rate for $\eta^\patchsymbol{}$}
\label{se:star}

In this section we recall the abstract theory developed in~\cite{MR2875241} for the analysis of
AFEM formulations where nonresidual estimators are used and we show how to use it for the analysis of the AFEM based on the Braess--Sch\"oberl error estimator. The interested reader is referred
to~\cite[Sections 4--6]{MR2875241} for all details of the theory. The main results, stated in Theorems~5.1
and~6.6, are the contraction property for the total error (which guarantees the convergence
of the AFEM procedure) and the quasioptimality of the rate of convergence in terms of number
of degrees of freedom.

As usual when dealing with adaptive schemes, we use a notation that takes into account the levels of refinement instead of the mesh size. We denote by $\tri{0}$ the initial triangulation of $\domain$ and by $u_\ell$ the discretization of $u$ on the triangulation $\tri{\ell}$ obtained from $\tri{0}$ after $\ell$ refinements. For some of the remaining notation we will adopt the one from~\cite{MR2875241}. 

\medskip
\noindent\textbf{Contraction property.}
If $u$ is the solution of problem~\eqref{eq:variational} and $u_j$ is the solution of the
corresponding discrete problem after $j$ refinements, the contraction property states the existence
of constants $\gamma>0$, $0<\alpha<1$, and $\mathcal{J}\in\mathbb{N}$ such that
\begin{equation}
\vertiii{u-u_{j+\mathcal{J}}}_\Omega^2+\gamma\osc^2_{\tri{j+\mathcal{J}}}(u_{j+\mathcal{J}},\tri{j+\mathcal{J}})
\le\alpha^2
\left(\vertiii{u-u_{j}}_\Omega^2+\gamma\osc^2_{\tri{j}}(u_{j},\tri{j})\right),
\label{eq:contraction}
\end{equation}
where the norm $\vertiii{\cdot}$ denotes the $H^1$-seminorm (equivalent to the norm in $H^1_0(\Omega)$).
The main difference with respect to the standard contraction property commonly used in this context
is that in general there might not be a contraction between two consecutive refinement levels $j$ and
$j+1$, but contraction is guaranteed every $\mathcal{J}$ levels.

\medskip
\noindent\textbf{Quasioptimal decay rate.}
The quasioptimality in terms of degrees of freedom is described as usual in the framework of
approximation classes. The triple $(u,f,\data)$, of the solution, the right hand side, and the other
data of problem~\eqref{eq:variational}, is in the approximation class $\AS$ if
\[
|(v,f,\data)|_\AS:=\sup_{N>0}(N^s\toterr{N}{v}{f}{\data})<\infty,
\]
where the total error $\toterr{N}{v}{f}{\data}$, in the set $\mathbb{T}_N$ of conforming triangulations
generated from $\tri{0}$ with at most $N$ elements more than $\tri{0}$, is defined as
\begin{equation}
\toterr{N}{v}{f}{\data}=
\inf_{\tri{}\in\mathbb{T}_N}
\inf_{V\in \Vhexpl}
(\vertiii{u-V}_\Omega^2
+\osc^2_{\tri{}}
(V,\tri{}))^{1/2}.
\label{eq:toterr}
\end{equation}

With this notation, the quasioptimal decay rate is expressed by the following formula
\begin{equation}
\vertiii{u-u_j}_{\Omega}+\osc_{\tri{j}}(u_j,\tri{j})\le C(\#\tri{j}-\#\tri{0})^{-s}|(v,f,\data)|_\AS,
\label{eq:optimality}
\end{equation}
where the constant $C$ is independent of $j$. We refer the interested reader to~\cite{MR2875241} for more
detail on the constant $C$, especially for its dependence on $s$.
Clearly, $C$ will depend in particular on the initial triangulation $\tri{0}$ and on the integer
$\mathcal{J}$ appearing in the above contraction property.

The assumptions needed in order to get~\eqref{eq:contraction} and~\eqref{eq:optimality} are divided into
three main groups: assumptions related to the \emph{a posteriori error estimators}, assumptions related to
the \emph{oscillations}, and assumptions related to the design of the \emph{adaptive finite element
method}. We are going to use the \emph{newest vertex bisection} algorithm for the refinement of the mesh (see, for instance,~\cite{MR2353951}).
While assumptions on oscillations and on the design of AFEM do not change when residual or nonresidual
a posteriori estimators are used, the main modification for the analysis of nonresidual estimators is given
by the verification of the assumptions related to the a posteriori error estimators. For this reason, we
focus in this section only on these assumptions (see~\cite[Assumption 4.1]{MR2875241}), which are the main
object of our analysis in the present paper. We will also make more precise the reduction assumption about the oscillations (see condition [H5] later on).
We adopt the notation of the previous section and we state
the assumptions for a generic error estimator $\eta(u_\ell,\tri{})$. In~\cite{MR2875241} there are some typos ($V$ instead of $U$, for instance) that we have corrected here.

\cite{MR2875241} considers a closed set called $K$-element made of elements or sides and denoted by $\mathcal K_{\tri{}}$. We restrict to the case when $K$ is a triangle. The following definition of refined set of order $j$ is needed between to (not necessarily consecutive) meshes $\tri{\ell}$ and $\tri{m}$
\begin{align*}
    \Rset{j} =\{
    T \in \tri{\ell}:  
    \min\limits_{
    T^\prime  \in \tri{m}\text{ and }
    T^\prime \subset  T  }
    \left( g(T^\prime)-g(T) \right) \geq j
    \},
\end{align*}
where the generation $g(T)$ of $T \in \tri{}$ is the number of bisections needed to create $T$ from the initial triangulation $\tri{0}$.

The four assumptions related to the
a posteriori error estimator state the existence of four constants $C_{re}$, $C_{ef}$, $C_{dre}$, and
$C_{def}$ and of an index $j^\star$ such that the following four conditions are satisfied.

\begin{description}

\item[{[H1] Global upper bound (reliability)}]

\[
\vertiii{u-u_l}^2_\Omega\le C_{re}(\eta(u_l,\tri{})^2+\osc_{\tri{}}(u_l,\tri{})^2).
\]

\item[{[H2] Global lower bound (efficiency)}]

\[
\eta(u_l,\tri{})^2\le C_{ef}(\vertiii{u-u_l}^2_\Omega+\osc_{\tri{}}(u_l,\tri{})^2).
\]

\item[{[H3] Localized upper bound (discrete reliability)}]

\[
\vertiii{u_m-u_l}_\Omega^2\le C_{dre}(\eta_{\tri{}}(u_l,\Rset{1})^2+\osc_{\tri{}}(u_l,\Rset{j^\star})^2).
\]

\item[{[H4] Discrete local lower bound (discrete efficiency)}]

\[
\eta_{\tri{}}(u_l,\Rset{j^\star})^2\le C_{def}(\vertiii{u_m-u_l}_\Omega^2+\osc_{\tri{}}(u_l,\Rset{j^\star})^2).
\]

\end{description}

In particular, it is clear that conditions H1 and H2 are satisfied by the estimators we are considering (see~\eqref{eq:BSrel} and~\eqref{eq:BSeff})

\begin{Remark}
Actually, in~\cite{MR2875241} the conditions H3 and H4 are stated with $\osc_{\tri{}}(u_l,\Rset{1})$ instead of $\osc_{\tri{}}(u_l,\Rset{j^\star})$. In our case, for technical reasons that will be apparent soon, we have to use $j^\star$ levels of refinements for the oscillations as well. The proof presented in~\cite{MR2875241} carries over to this situation with the natural modifications.
\end{Remark}

In H1-H4, particular attention has to be paid to the oscillation terms.
When we are using polynomials of degree $k$ for the solution of the discrete problem~\eqref{eq:variational_h}, we usually define the oscillation terms by introducing the projection $\Pi_{k-1}$ onto polynomials of degree $k-1$. The standard oscillation term would then read
\begin{align}
\osc(f,{\tri{}})  =
\left(
\sum
\limits_{T\in \mathcal T}
\| h_{T}(f-\Pi^{k-1}_T f) \|_{0,T}^2
\right)^{1/2}.
\end{align}
On the other hand, we will consider the estimator $\eta^{\patchsymbol{}}$ built on patches and for this reason it makes sense to introduce a corresponding definition of patch oscillations:
\begin{align}
\osc^\patchsymbol(f,{\tri{}})  =
\left(
\sum
\limits_{\nu \in \mathcal V_T}
\osc(f,\nu)^2
\right)^{1/2},
\end{align}
where
\begin{align}
\osc(f,\nu)  =
\left(
\sum
\limits_{T\in  \omega_\nu}
\| h_{T}(f-\Pi^{k-1}_T f) \|_{0,T}^2
\right)^{1/2}.
\end{align}

The critical assumption related to the oscillations (see~\cite[Assuption~4.2(a)]{MR2875241}) is the following one.

\begin{description}
\item[{[H5] Oscillation reduction}]
there exists a constant $\lambda\in]0,1[$ such that
\[
    \osc_{\tri{m}} (f,\tri{m})^2
    \leq
    \osc_{\tri{l}} (f,\tri{l})^2
    -
    \lambda
    \osc_{\tri{l}} (f,\Rset{{j^\star}})^2.
\]
\end{description}
\begin{figure}[h]
\begin{minipage}[c]{0.49\textwidth}
\begin{tikzpicture} [scale=1.5]
\foreach \angle in {0,60,...,300} {
\begin{scope}[rotate=\angle]
  \draw[thick] (0,0) -- ++(0:1) -- ++(120:1) -- ++(0:1) -- ++(-120:1)  -- ++(0:1) -- ++(120:2);
  \end{scope}
}
\filldraw[blue,thick,opacity=0.3] (-1,0) --
(-1.5,{sqrt(0.75)})--
(-1,{sqrt(3)}) -- (1,{sqrt(3)})--(2,0)
-- (1.5,-{sqrt(0.75)}) --
(-0.5,-{sqrt(0.75)})-- (-1,0)
;
\filldraw[red,thick,opacity=0.3] (0,0) -- (1,0)--(0.5,{sqrt(0.75)}) -- (-0.5,{sqrt(0.75)});
\draw[red,thick] (1,0) --(-0.5,{sqrt(0.75)});
\filldraw [red] (0,0) circle(1pt);
\filldraw [red] (0.5,{sqrt(0.75)}) circle(1pt);

\filldraw [red] (1,0) circle(1pt);
\filldraw [red] (-0.5,{sqrt(0.75)}) circle(1pt);

\filldraw [green] (0,{sqrt(0.75)}) circle(1pt);
\filldraw [green] (-0.25,{sqrt(0.1875)}) circle(1pt);

\filldraw [green] (-1,0) circle(1pt);
\filldraw [green] (0,{sqrt(3)}) circle(1pt);

\draw[green,thick,] (-1,0) --(-0.25,{sqrt(0.1875)})--
(0.25,{sqrt(0.1875)});

\draw[green,thick,] (0,{sqrt(3)}) --(0,{sqrt(0.75)})--
(0.25,{sqrt(0.1875)});

\filldraw [blue] (0.25,{sqrt(0.1875)}) circle(1pt);


\end{tikzpicture}
\caption{$\mathcal S^0=\mathcal S^1=\mathcal S^2$}
\label{fig:tri1}
\end{minipage}
\begin{minipage}[c]{0.49\textwidth}
\begin{tikzpicture} [scale=1.5]
\foreach \angle in {0,60,...,300} {
\begin{scope}[rotate=\angle]
  \draw[thick] (0,0) -- ++(0:1) -- ++(120:1) -- ++(0:1) -- ++(-120:1)  -- ++(0:1) -- ++(120:2);
  \end{scope}
}
\filldraw[blue,thick,opacity=0.3] (-1,0) --
(-1.5,{sqrt(0.75)})--
(-1,{sqrt(3)}) -- (1,{sqrt(3)})--(2,0)
-- (1.5,-{sqrt(0.75)}) --
(-0.5,-{sqrt(0.75)})-- (-1,0)
;
\filldraw[red,thick,opacity=0.3] (0,0) -- (1,0)--(0.5,{sqrt(0.75)}) -- (-0.5,{sqrt(0.75)});
\draw[red,thick] (1,0) --(-0.5,{sqrt(0.75)});
\filldraw [red] (0,0) circle(1pt);
\filldraw [red] (0.5,{sqrt(0.75)}) circle(1pt);

\filldraw [red] (1,0) circle(1pt);
\filldraw [red] (-0.5,{sqrt(0.75)}) circle(1pt);

\filldraw [green] (0.5,0) circle(1pt);
\filldraw [green] (0,{sqrt(0.75)}) circle(1pt);
\filldraw [green] (-0.25,{sqrt(0.1875)}) circle(1pt);
\filldraw [green] (0.75,{sqrt(0.1875)}) circle(1pt);

\filldraw [green] (-1,0) circle(1pt);
\filldraw [green] (1.5,{sqrt(0.75)}) circle(1pt);
\filldraw [green] (0,{sqrt(3)}) circle(1pt);
\filldraw [green] (0.5,-{sqrt(0.75)}) circle(1pt);

\draw[green,thick,] (-1,0) --(-0.25,{sqrt(0.1875)})--
(0.75,{sqrt(0.1875)})-- (1.5,{sqrt(0.75)});

\draw[green,thick,] (0,{sqrt(3)}) --(0,{sqrt(0.75)}) --
(0.5,0)-- (0.5,-{sqrt(0.75)});

\filldraw [blue] (0.25,{sqrt(0.1875)}) circle(1pt);

\draw[yellow,thick,]  (-0.25,{sqrt(0.1875)})
--(0,{sqrt(0.75)})--(0.75,{sqrt(0.1875)})--(0.5,0)
--(-0.25,{sqrt(0.1875)});

\filldraw [yellow] (0.125,{sqrt(0.046875)}) circle(1pt);
\filldraw [yellow] (0.375,{sqrt(0.1875*0.25*9)}) circle(1pt);
\filldraw [yellow] (0.625,{sqrt(0.1875*0.25)}) circle(1pt);
\filldraw [yellow] (-0.125,{sqrt(0.1875*0.25*9)}) circle(1pt);
\end{tikzpicture}
\caption{$\mathcal S^0 =\mathcal S^3$}
\label{fig:tri1b}
\end{minipage}
\end{figure}

\newcommand{\myRset}[1]{\mathcal R^{#1}_{\trito{\ell}{\ell+1}}}
\begin{Remark}
We need to modify the original assumption of~\cite{MR2875241} by replacing $\osc_{\tri{l}} (f,\Rset{{1}})$ with $\osc_{\tri{l}} (f,\Rset{{j^\star}})$. A simple example for the necessity of this modification is to consider a triangulation $\mathcal T_\ell$ and the triangulation $\mathcal T_{\ell+1}$ obtained with a minimal refinement, so that only two triangles belong to $\myRset{1}$. This refinement is marked in red in Figure~\ref{fig:tri1} and we can see that \begin{align*}
    supp(\eta^\patchsymbol{}(u_h,\myRset{1})) = 
    supp(\eta^\patchsymbol{}(u_h,\{ T \in T_{\ell+1}, T\notin T_{\ell}\}).
\end{align*}
Repeating this argument, we have also for some $k>1$ that
\begin{align*}
supp(\eta^\patchsymbol{}(u_h,\myRset{1})) = 
    supp(\eta^\patchsymbol{}(u_h,\{ T \in T_{\ell+k}, T\notin T_{\ell}\}).
\end{align*}
This is illustrated in Figure~\ref{fig:tri1b} with the green refinement leading to the set $\myRset{2}$. The same holds for $\myRset{3}$ (yellow refinement). We see in Figure~\ref{fig:tri2}, with the notation 
$\mathcal S^k:=supp(\eta^\patchsymbol{}(u_h,\myRset{k}))$, that only the fourth refinement leads to a reduction of the support of the estimator.
\end{Remark}

In the rest of this section we are going to show that hypotheses [H1-H4] hold true for $\eta^{\patchsymbol}$ and $\osc^\patchsymbol$, in the case when $d=2$, with $j^\star$ defined in the following lemma.

\begin{Lemma}
\label{le:j*}
Assume that the triangulation $\tri{\ell}$ is shape-regular. Let $\tri{m}$ be a triangulation obtained from $\tri{\ell}$ after $m-\ell$ refinements with the newest vertex bisection strategy (see, for instance,~\cite{MR2353951}).
Then there exists $j^\star$ such that $\Rset{j^\star}$ satisfies the following property: all triangles in $\omega_\nu$, for all $\nu\in\mathcal{V}_T$, and all their edges have an interior node that is a vertex of a triangle of $\tri{m}$.
\end{Lemma}
\begin{proof}
Let $n^\star$ denote the maximum number of triangles in a patch in the triangulation $\tri{\ell}$. The shape-regularity of $\tri{\ell}$ implies that $n^\star$ is bounded.

We observe that if $T\in\Rset{2}$ then the adjacent triangles of $T$ belong at least to $\Rset{1}$; this is illustrated in Figure~\ref{fig:R2}.
If moreover $T\in\Rset{3}$ then $T$ has the interior node property, but two of the adjacent triangles could still belong only to $\Rset{1}$ as is it shown in Figure~\ref{fig:R3}.

\begin{figure}
\begin{minipage}[c]{0.49\textwidth}
\begin{tikzpicture} [scale=2]
\foreach \angle in {0,60,...,300} {
\begin{scope}[rotate=\angle]
  \draw[thick] (0,0) -- ++(0:1) -- ++(120:1) -- ++(0:1) -- ++(-120:1)  -- ++(0:1) -- ++(120:2);
  \end{scope}
}

\filldraw[blue,thick,opacity=0.3] (-0.75,{sqrt(0.1875)}) --
(-1.5,{sqrt(0.75)})--
(-1,{sqrt(3)}) -- (-0.25,{3*0.25*sqrt(3)})--(0,{sqrt(3)})--(0.25,{3*0.25*sqrt(3)})--(1,{sqrt(3)})--(1.5,{0.5*sqrt(3)})
-- (0.75,{0.25*sqrt(3)}) --(0.5,0)--
(0.5,-{sqrt(0.75)})--(-0.5,-{sqrt(0.75)})-- (-0.5,0)--(-1,0)
;
\filldraw[red,thick,opacity=0.3] (0,0) -- (1,0)--(0.5,{sqrt(0.75)}) -- (-0.5,{sqrt(0.75)});

\draw [orange] --
(-1.5,{sqrt(0.75)});

\draw [orange] (-0.5,0)--
(-0.5,-{sqrt(0.75)});

\draw [orange] (-0.25,{sqrt(0.1875*9)})--
(-1,{2*sqrt(0.75)});

\draw [orange] (0.25,{sqrt(0.1875*9)})--
(1,{2*sqrt(0.75)});

\draw [orange] (-0.125,{sqrt(0.1875*0.25)})--
(0.125,{sqrt(0.1875*0.25)})--
(-0.25,{sqrt(0.1875*4)})
--(-0.25,{sqrt(0.1875*9)})--
(0,{sqrt(0.1875*4)})--
(0.25,{sqrt(0.1875*9)})--
(0.25,{sqrt(0.1875*4)})--(0.375,{sqrt(0.1875*0.25*9)})--(-0.375,{sqrt(0.1875*0.25*9)})--
(-0.75,{sqrt(0.1875)})--(-0.25,{sqrt(0.1875)})
--(-0.5,0)--(-0.125,{sqrt(0.1875)*0.5});

\draw[red,thick] (1,0) --(-0.5,{sqrt(0.75)});
\filldraw [red] (0,0) circle(1pt);
\filldraw [red] (0.5,{sqrt(0.75)}) circle(1pt);

\filldraw [red] (1,0) circle(1pt);
\filldraw [red] (-0.5,{sqrt(0.75)}) circle(1pt);

\filldraw [green] (0.5,0) circle(1pt);
\filldraw [green] (0,{sqrt(0.75)}) circle(1pt);
\filldraw [green] (-0.25,{sqrt(0.1875)}) circle(1pt);
\filldraw [green] (0.75,{sqrt(0.1875)}) circle(1pt);

\filldraw [green] (-1,0) circle(1pt);
\filldraw [green] (1.5,{sqrt(0.75)}) circle(1pt);
\filldraw [green] (0,{sqrt(3)}) circle(1pt);
\filldraw [green] (0.5,-{sqrt(0.75)}) circle(1pt);

\draw[green,thick,] (-1,0) --(-0.25,{sqrt(0.1875)})--
(0.75,{sqrt(0.1875)})-- (1.5,{sqrt(0.75)});

\draw[green,thick,] (0,{sqrt(3)}) --(0,{sqrt(0.75)}) --
(0.5,0)-- (0.5,-{sqrt(0.75)});

\filldraw [blue] (0.25,{sqrt(0.1875)}) circle(1pt);

\draw[yellow,thick,]  (-0.25,{sqrt(0.1875)})
--(0,{sqrt(0.75)})--(0.75,{sqrt(0.1875)})--(0.5,0)
--(-0.25,{sqrt(0.1875)});

\filldraw [yellow] (0.125,{sqrt(0.046875)}) circle(1pt);
\filldraw [yellow] (0.375,{sqrt(0.1875*0.25*9)}) circle(1pt);
\filldraw [yellow] (0.625,{sqrt(0.1875*0.25)}) circle(1pt);
\filldraw [yellow] (-0.125,{sqrt(0.1875*0.25*9)}) circle(1pt);

\filldraw [orange]
(-1.5,{sqrt(0.75)}) circle(1pt);

\filldraw [orange] (-0.5,0) circle(1pt);
\filldraw [orange](-0.5,-{sqrt(0.75)})circle(1pt);

\filldraw [orange] (-0.25,{sqrt(0.1875*9)}) circle(1pt);
\filldraw [orange](-1,{2*sqrt(0.75)})circle(1pt);

\filldraw [orange] (0.25,{sqrt(0.1875*9)}) circle(1pt);
\filldraw [orange] 
(1,{2*sqrt(0.75)}) circle(1pt);

\filldraw [orange] (-0.125,{sqrt(0.1875*0.25)})circle(1pt);
\filldraw [orange]
(0.125,{sqrt(0.1875*0.25)})circle(1pt);
\filldraw [orange]
(-0.25,{sqrt(0.1875*4)})circle(1pt);
\filldraw [orange](-0.25,{sqrt(0.1875*9)})circle(1pt);
\filldraw [orange]
(0,{sqrt(0.1875*4)})circle(1pt);
\filldraw [orange]
(0.25,{sqrt(0.1875*9)})circle(1pt);
\filldraw [orange]
(0.25,{sqrt(0.1875*4)})circle(1pt);
\filldraw [orange] (0.375,{sqrt(0.1875*0.25*9)})circle(1pt);
\filldraw [orange]
(-0.375,{sqrt(0.1875*0.25*9)})circle(1pt);
\filldraw [orange]
(-0.75,{sqrt(0.1875)})circle(1pt);
\filldraw [orange](-0.25,{sqrt(0.1875)})circle(1pt);

\filldraw [orange](-0.5,0)circle(1pt);
\filldraw [orange](-0.125,{sqrt(0.1875)*0.5})
circle(1pt);



\filldraw [orange] (-0.75,{sqrt(0.1875)}) circle(1pt);
\filldraw [orange] (-0.5,0) circle(1pt);

\end{tikzpicture}
\caption{$\mathcal S^4 \subsetneq \mathcal S^0$}
\label{fig:tri2}
\end{minipage}
\end{figure}
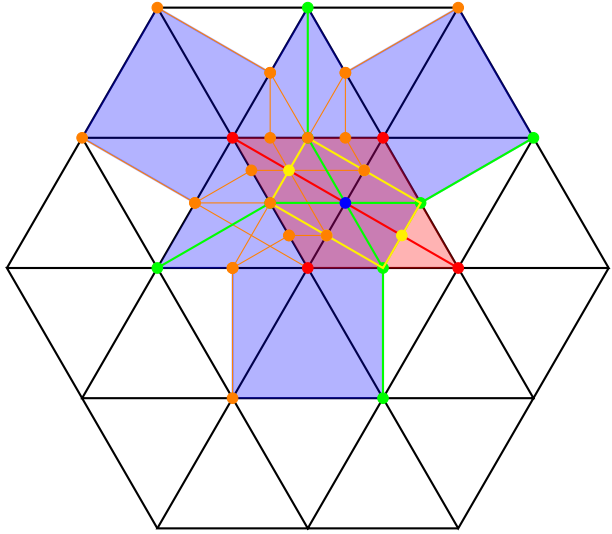
However, $T \in \Rset{4}$ implies that the two adjacent triangles belong at least to $\Rset{2}$ (see Figure~\ref{fig:R4}).
Similarly, $T\in\Rset{6}$ implies that the two adjacent triangles belong to $\Rset{3}$.
Repeating this argument shows that for $j^\star = 3n^\star/4$ all triangles in the patch have the interior node property, and all facets have an interior node.
\end{proof}

We are now showing that conditions [H1-H4] hold true for the residual error estimator defined on patches $\tilde\eta^{res}$; thanks to the equivalence proved in Section~\ref{se:equiv} the same conditions will hold for the Braess--Sch\"oberl estimator $\eta^{\patchsymbol{}}$ as well.

\begin{Lemma}[H3 --- discrete reliability for $\tilde\eta^{res}$]
Let $\tri{m}$ be a refinement of $\tri{\ell}$. Then
\[
\|| u_\ell-u_m \|| \lesssim 
    \tilde\eta^{res}
    (u_\ell,\Rset{1})+
    \osc^\patchsymbol(u_\ell,\Rset{1})
\]
\label{le:H3}
\end{Lemma}
\begin{proof}
It is well-known (see~\cite[Theorem~4.1]{MR2324418}) that the discrete reliability properties holds true for the standard residual error estimator, that is
\[
\|| u_\ell-u_m \|| \lesssim 
    \eta^{res}
    (u_\ell,\Rset{1})+
    \osc^\patchsymbol(u_\ell,\Rset{1}).
\]
Clearly, the extension to $\tilde\eta^{res}$ is straightforward.
\end{proof}

\begin{Lemma}[H4 --- discrete efficiency for $\tilde\eta^{res}$]
Let $\tri{m}$ be a refinement of $\tri{\ell}$ and let $j^\star$ be the index introduced in Lemma~\ref{le:j*}. Then it holds
\[
\tilde\eta^{res} (u_\ell,\Rset{j^\star}) \lesssim 
  \|| u_\ell -u_m \|| +\osc^\patchsymbol(u_\ell,\Rset{1}).
\]
\label{le:H4}
\end{Lemma}

\begin{proof}
From the definition of $j^\star$ we have that if $T$ belongs to $\Rset{j^\star}$ then $T$ and its edges have the interior node property. Moreover, $ \tilde\omega_T:=\{ T^\prime \ : \ T^\prime\cap \omega_T \neq 0 \}$ is contained in $\Rset{1}$.

Therefore, we can use the fact that the standard residual estimator $\eta^{res}$ is discretely efficient, that is,
\[
    \eta^{res}(u_\ell,T) \lesssim\vertiii{u_\ell - u_m}_{\omega_T} +
    \vertiii{(id-\Pi^{k-1}_T) f}_{\omega_T}.
\]
It follows
\begin{align*}
    &\left(\tilde\eta^{res}(u_\ell,\Rset{j^\star})\right)^2
    =\sum\limits_{T \in \Rset{j^\star}}
    \sum\limits_{\nu \in \mathcal V_\nu }
    (\eta^{res}(u_\ell,\omega_\nu)   )^2
    \\
    &\qquad=\sum\limits_{T \in \Rset{j^\star}}
    \sum\limits_{\nu \in \mathcal V_\nu }
    \sum\limits_{T^\prime \in \omega_\nu}
    \left((\eta^{res}(u_\ell,T^\prime)   )^2
     +
    \| h_T(id- \Pi^{k-1}_T)(f) \|_{T^\prime}^2\right)\\
    &\qquad \lesssim\sum\limits_{T \in \Rset{j^\star}}
    \sum\limits_{\nu \in \mathcal V_\nu }
    \sum\limits_{T^\prime \in \omega_\nu}
    \left(\vertiii{u_\ell-u_m}_{T^\prime}
     +
    \| h_T(id- \Pi^{k-1}_T)(f) \|_{\omega_{T^\prime}}^2\right)
    \\
   &\qquad\lesssim
     \vertiii{u_\ell -u_m}_{\Omega} +
     \sum\limits_{T \in \Rset{j^\star}}
     \osc^\patchsymbol(u_\ell,\tilde \omega_T)\\
    &\qquad\lesssim
     \vertiii{u_\ell -u_m}_{\Omega} +\osc^\patchsymbol(u_\ell,\Rset{1}) .
\end{align*}
\end{proof}
\begin{figure}
\begin{minipage}[l]{0.4\textwidth}
\begin{tikzpicture} 
\draw[gray, thick] (0,0) -- (4,0) -- (2.5,2) -- (0,0); 
\draw[red, thick] (1.25,1) -- (2,0); 
\draw[red, thick] (2,0) -- (2.5,2); 
\draw[red, thick] (3.25,1) -- (2,0); 
\end{tikzpicture}
\caption{$T\in\Rset{2}$}
\label{fig:R2}
\end{minipage}
\begin{minipage}[l]{0.4\textwidth}
\begin{tikzpicture} 
\draw[gray, thick] (0,0) -- (4,0) -- (2.5,2) -- (0,0); 
\draw[red, thick] (1.25,1) -- (2,0); 
\draw[red, thick] (2,0) -- (2.5,2); 
\draw[red, thick] (3.25,1) -- (2,0); 
\draw[blue, thick] (1.25,1) -- (3.25,1); 
\draw[blue, thick] (1,0) -- (1.25,1); 
\draw[blue, thick] (3,0) -- (3.25,1); 
\end{tikzpicture}
\caption{$T\in\Rset{3}$}
\label{fig:R3}
\end{minipage}
\begin{minipage}[l]{0.4\textwidth}
\begin{tikzpicture} 
\draw[gray, thick] (0,0) -- (4,0) -- (2.5,2) -- (0,0); 
\draw[red, thick] (1.25,1) -- (2,0); 
\draw[red, thick] (2,0) -- (2.5,2); 
\draw[red, thick] (3.25,1) -- (2,0); 
\draw[blue, thick] (1.25,1) -- (3.25,1); 
\draw[blue, thick] (1,0) -- (1.25,1); 
\draw[blue, thick] (3,0) -- (3.25,1); 
\draw[green, thick] (0.625,0.5) -- (1,0); 
\draw[green, thick] (1,0) -- (2.825,1.5); 
\draw[green, thick] (3,0) -- (1.825,1.5); 
\draw[green, thick] (3.625,0.5) -- (3,0);
\end{tikzpicture}
\caption{$T\in\Rset{4}$}
\label{fig:R4}
\end{minipage}
\end{figure}
The next lemma is related to the oscillation reduction stated in condition [H5]. For completeness, we show the condition met both by the standard oscillation term and by the patchwise oscillation; in our analysis we are going to use the latter one.

\begin{Lemma}[H5 --- Oscillation reduction]
Let $\tri{m}$ be a refinement of $\tri{\ell}$ and let $j^\star$ be the index introduced in Lemma~\ref{le:j*}. Then it holds
\begin{align}
    \osc_{\tri{m}} (f,\tri{m})^2
    \leq
    \osc_{\tri{l}} (f,\tri{l})^2
    -
    \lambda
    \osc_{\tri{l}} (f,\Rset{1})^2
    \label{eq:osc1}
\end{align}
and
\begin{align}
    \osc^\patchsymbol_{\tri{m}} (f,\tri{m})^2
    \leq
    \osc^\patchsymbol_{\tri{l}} (f,\tri{l})^2
    -
    \lambda
    \osc^\patchsymbol_{\tri{l}} (f,\Rset{{j^\star}})^2.
    \label{eq:osc2}
\end{align}
\label{le:osc}
\end{Lemma}
\begin{proof}
The first statement is equivalent to 
\begin{align}
    \osc_{\tri{m}} (f,\tri{m}^\star)^2
    \leq
    (1
    -
    \lambda)
    \osc_{\tri{l}} (f,\Rset{1})^2
\end{align}
where $\tri{m}^\star = \tri{m}\backslash \tri{l}$. 

\newcommand{\T}{\mathfrak{ T}}
Consider a triangle $T_m$ in $\tri{m}^\star$ which originates from the triangle $\T_\ell(T_m)$ in $\tri{\ell}$. 
Our refinement strategy guarantees that the mesh size is reduced, so that $h_{T_m}\le\gamma h_{\T_\ell(T_m)}$ for a positive $\gamma<1$.
Then, it holds
\begin{align*}
    \osc_{\tri{m}} (f,\tri{m}^\star)^2
    &= 
    \sum\limits_{T_m \in \tri{m}^\star}
    \| h_{T_m} (f-\Pi^{k-1}_{T_m} f) \|^2_{0,T_m}\\&
    \leq 
    \sum\limits_{T_m \in \tri{m}^\star}
    (\gamma h_{\T_l(T_m)})^2 \|  (f-\Pi^{k-1}_{T_m} f) \|^2_{0,T_m}\\&
    =
    \sum\limits_{T_l \in \Rset{1}}
    \sum\limits_{T_m \in \tri{m}^\star , \T_l(T_m) = T_l }
    (\gamma h_{\T_l(T_m)})^2 \|  (f-\Pi^{k-1}_{T_m} f) \|^2_{0,T_m}\\&
    \leq
    \sum\limits_{T_l \in \Rset{1}}
    (\gamma h_{T_l})^2 \|  (f-\Pi^{k-1}_{T_l} f) \|^2_{0,T_l}
    \\&
    \leq
    \gamma^2 \osc_{\tri{l}} (f,\Rset{1})^2.
\end{align*}
So we have~\eqref{eq:osc1} with $\lambda = 1-\gamma^2$.

The second statement is equivalent to
\begin{align*}
 \osc^\patchsymbol_{\tri{m}} (f,\tri{m}^\star)^2
    &\leq
    \osc^\patchsymbol_{\tri{l}} (f,\Rset{1})^2
    -
    \lambda
    \osc^\patchsymbol_{\tri{l}} (f,\Rset{{j^\star}})^2 \\
    &\leq
    \osc^\patchsymbol_{\tri{l}} (f,\Rset{1} \backslash\Rset{{j^\star}})^2
    +(1
    -
    \lambda)
    \osc^\patchsymbol_{\tri{l}} (f,\Rset{{j^\star}})^2.
\end{align*}
Recall that the definition of $j^\star$ implies that for any $T$ in $\Rset{j^\star}$ all the triangles in $\omega_T$
belong to $\Rset{1}$. Therefore,
\begin{align*}
    &\osc^\patchsymbol_{\tri{m}} (f,\tri{m}^\star)^2
    = 
    \sum\limits_{T^\prime_m  \in \tri{m}^\star}
    \sum\limits_{T_m \in \omega_{T^\prime_m }}
    \| h_{T_m} (f-\Pi^{k-1}_{T_m} f) \|^2_{0,T_m}\\&\qquad
    \leq 
    \sum\limits_{T^\prime_m  \in \tri{m}^\star}
    \sum\limits_{T_m \in \omega_{T^\prime_m }}
    (\gamma h_{\T_l(T_m)})^2 \|  (f-\Pi^{k-1}_{T_m} f) \|^2_{0,T_m}\\&\qquad
    =
    \sum\limits_{T_l \in \Rset{1}}
    \sum\limits_{T^\prime_m \in \tri{m}^\star , \T_l(T^\prime_m) = T_l }
    \sum\limits_{T_m \in \omega_{T^\prime_m }}
    (\gamma h_{\T_l(T_m)})^2 \|  (f-\Pi^{k-1}_{T_m} f) \|^2_{0,T_m}\\&\qquad
    \leq
    \sum\limits_{T_l \in \Rset{1}\backslash \Rset{j^\star} }
    (\gamma h_{T_l})^2 \|  (f-\Pi^{k-1}_{T_l} f) \|^2_{0,T_l}
    +
    \gamma^2 \osc^\patchsymbol_{\tri{l}} (f,\Rset{j^\star})^2.
\end{align*}
\end{proof}
We are now in the position of stating our main result concerning the convergence of AFEM based on the Braess--Sch\"oberl error estimator.

\begin{Theorem}
Let $u$ be the solution of Problem~\eqref{eq:variational} and consider a SOLVE--ESTIMATE--MARK--REFINE strategy satisfying the following properties.
\begin{enumerate}
    \item In the solve module the solution is computed exactly. 
    \item The estimate module makes use of the Braess--Sch\"oberl error estimator $\eta^{\patchsymbol{}}$ defined on patches and takes into account the total error~\eqref{eq:toterr} with the patchwise oscillation term $\osc^\patchsymbol{}$.
    \item The mark module is the usual D\"orfler marking strategy.
    \item The refine module is performed using the newest vertex bisection algorithm and it is slightly modified from the standard routines, as described in~\cite{MR2875241}, using the iteration counter $j^\star$ defined in Lemma~\ref{le:j*}, so that the interior node property is satisfied.
\end{enumerate}
Then the sequence of discrete solutions $\{u_\ell\}$ converges to $u$ with the quasioptimal decay rate
\[
\|u-u_\ell\|_{1,\Omega}+\osc^{\patchsymbol{}}(u_\ell,\tri{\ell})\lesssim (\#\tri{\ell}-\#\tri{0})^{-s}|(v,f,\data)|_\AS
\]
(see~\eqref{eq:optimality}).
\label{th:mainstar}
\end{Theorem}

\begin{proof}

As explained above, we need to show that the five conditions H1-H5 are satisfied.
We have already observed that H1 (global reliability) and H2 (global efficiency) are proved in~\cite[Theorems 9.4 and 9.5]{Bra:13} (see~\eqref{eq:BSrel} and~\eqref{eq:BSeff}).

The equivalence between the estimator $\eta^{\patchsymbol{}}$ and the patchwise residual estimator $\tilde\eta^{res}$ (see Section~\ref{se:equiv}), together with Lemmas~\ref{le:H3} and~\ref{le:H4}, leads directly to the localized bounds H3 (discrete reliability) and H4 (discrete efficiency) for the estimator $\eta^{\patchsymbol{}}$. Finally, [H5] (oscillation reduction) has been proved in Lemma~\ref{le:osc} (see Equation~\eqref{eq:osc2}).
\end{proof}

\begin{Remark}
Another way to prove a result analogue to the one presented in Theorem~\ref{th:mainstar} would be to use the theory developed in~\cite{MR2776915}. In such theory, the refine module is not modified from the standard routines, while the mark module acts on patches instead of on single elements. Unfortunately, the theory of~\cite{MR2776915} assumes that the oscillations are dominated by the error estimator, which might not be true in our case; a possible fix would be the use of a separate marking strategy as in~\cite{MR3719030}.
\end{Remark}

\section{Optimal convergence rate for $\eta^\Delta$}
\label{se:optimalfleur}

In this section we see how the results of the previous section can be extended to $\eta^\Delta$, which is the error estimator usually referred to as Braess--Sch\"oberl estimator.

%

Even if the estimator is constructed element by element, we keep using the oscillation term $\osc^\patchsymbol{}(u_\ell,\tri{\ell})$ defined on patches of elements. This is needed, in particular, for the proof of the discrete efficiency (see Lemma~\ref{le:eureka}).

It is clear from the above discussion that, in order to apply the theory of~\cite{MR2875241}, the two crucial properties are H3 (discrete reliability) and H4 (discrete efficiency). We are not going to use the equivalence with any residual-type error estimator, but we are showing these properties directly in the next two lemmas.

\begin{Lemma}[Discrete Reliability]
Let $\tri{m}$ be a refinement of $\tri{\ell}$, then
\[
\vertiii{u_\ell -u_m}\lesssim 
 \eta^\Delta (u_\ell,\Rset{1})
  +\osc_{\tri{}}(u_\ell,\Rset{1}).
\]
\end{Lemma}
\begin{proof} 
Since $u_m$ is the solution of~\eqref{eq:variational_h} on $\tri{m}$ and $u_m-u_l$ is piecewise polynomial of degree $k$ on $\tri{m}$ as well, we have 
\[
(\nabla(u_m-u_\ell),\nabla(u_m-u_\ell)) = (\nabla(u-u_\ell),\nabla(u_m-u_\ell)).
\]
Since $u_\ell$ is the solution of~\eqref{eq:variational_h} on $\tri{\ell}$, we have
\begin{align*}
    \vertiii{u_m-u_\ell}^2
    =
    (\nabla(u_\ell-u),\nabla(u_m-u_\ell))
    =
    (\nabla(u_\ell-u),\nabla(u_m-\mathcal {I}_{\tri{\ell}} u_m)),
\end{align*}
where $\mathcal {I}_{\tri{\ell}}$ is the Lagrange interpolation operator with respect to the triangulation $\tri{\ell}$ .
From the results of~\cite{BraPilSch:09} we obtain
\begin{align*}
    \vertiii{u_m-u_\ell}^2
    =
    -\langle
    R(u_\ell),u_m-\mathcal {I}_{\tri{\ell}} u_m
    \rangle
     =
    -(f - \Pi^{k-1}f+
    \qdelta{},\nabla (u_m-\mathcal {I}_{\tri{\ell}} u_m)
    ) .
    \end{align*}
Outside the refined set $\Rset{1}$ we have $u_m=\mathcal {I}_{\tri{}} u_m$. This include the boundary  $\partial\Rset{1}$, 
so that
$\nabla u_m=\nabla \mathcal {I}_{\tri{\ell}} u_m$ on $\Omega \backslash \Rset{1}$.
Therefore,
\begin{align*}
    \vertiii{u_m-u_\ell}^2
     &=
     \sum\limits_{T \in \Rset{1}}
    -( f - \Pi^{k-1}f+
    \qdelta{},\nabla (u_m-\mathcal {I}_{\tri{\ell}} u_m)
    )_T \\
    &\leq
    \left(
    \left({
     \sum\limits_{T \in \Rset{1}}
    \| \qdelta{}
    \|^2_T}\right)^{1/2}
    +
    \left({
     \sum\limits_{T \in \Rset{1}}
    \| f - \Pi^{k-1}_T f
    \|^2_T}\right)^{1/2}
    \right)\\& \qquad\qquad\times
    \left({
    \sum\limits_{T \in \Rset{1}}
    \| \nabla (u_m-u_\ell+u_\ell-\mathcal {I}_{\tri{\ell}} u_m)\|^2_T
    }\right)^{1/2}.
\end{align*}
From the identity
\[
\| \nabla (u_m-u_\ell+u_\ell-\mathcal {I}_{\tri{\ell}} u_m)\|^2_T =
    \| \nabla 
    (u_m-u_{\tri{\ell}}-
    \mathcal {I}_{\tri{\ell}}(u_m-u_\ell))\|^2_T \leq \| \nabla (u_m-u_\ell)\|^2_T
\]
we obtain
\begin{align*}
    \vertiii{u_m-u_\ell}^2
     &\leq  
\left(\eta(u_\ell,\Rset{1})
+
\osc_{\tri{}}(u_\ell,\Rset{1})
\right)
    \left({
    \sum\limits_{T \in \Rset{1}}
    \| \nabla (u_m-u_\ell)\|^2_T
    }\right)^{1/2} \\&
    \leq
    \left(\eta(u_\ell,\Rset{1})
    +
    \osc_{\tri{}}(u_\ell,\Rset{1})\right)
    \| \nabla (u_m-u_\ell)\|_\Omega.
\end{align*}
    Dividing by $\| \nabla (u_m-u_l)\|_\Omega$  finishes the proof.

\end{proof}

\begin{Lemma}[Discrete Efficiency]
Let $\tri{m}$ be a refinement of $\tri{\ell}$, then
\[
\eta^{\Delta}(u_l,\Rset{j^\star})\lesssim
\vertiii{u_l -u_m}+\osc^\patchsymbol{}_{\tri{}}(u_l,\Rset{1}),
\]
where $j^\star$ is defined in Lemma~\ref{le:j*}.
\label{le:eureka}
\end{Lemma}
\begin{proof}
This result is a consequence of the following inequality 
\[
\eta^\Delta(u_l,\Rset{j^\star})^2=\sum_{T\in\Rset{j^\star}}(\eta_T^\Delta)^2\le \eta^\patchsymbol_{\tri{}}(u_l,\Rset{j^\star})^2
\]
and of the analogous result for the patchwise estimator $\eta^\patchsymbol{}$.
\end{proof}

We have then proved all the conditions that allow us to state a theorem analogue to~\ref{th:mainstar} in the case of the standard estimator $\eta^\Delta$.

\begin{Theorem}
Let $u$ be the solution of~\eqref{eq:variational} and consider the adaptive strategy as in the Theorem~\ref{th:mainstar} with the standard Braess--Sch\"oberl error estimator $\eta^\Delta$ and the oscillation term $\osc^\patchsymbol{}$. Then the sequence of discrete solutions $\{u_\ell\}$ converges with the quasioptimal decay rate
\[
\|u-u_\ell\|_{1,\Omega}+\osc(u_\ell,\tri{\ell})\lesssim (\#\tri{\ell}-\#\tri{0})^{-s}|(v,f,\data)|_\AS.
\]
(see~\eqref{eq:optimality}).
\label{th:maindelta}
\end{Theorem}

Before concluding this section, we would like to briefly comment on the elasticity problem~\eqref{elasticity} for which the Prager--Synge theory has been developed.
In that case the symmetric gradients of the constitutive equation give an additional term in the integration by parts needed for the Prager--Synge Theorem~\ref{th:PS}. The anti-symmetric part of the equilibrated stress has therefore to be controlled. Clearly, symmetric $H(\div)$-conforming stress spaces such as the Arnold--Winther elements (see~\cite{ArnWin:02}) can be used as in~\cite{NicWitWoh:08} or~\cite{AinRan:10}. Another possibility is to impose the symmetric condition in a weak form~\cite{BerKobMolSta:19}. For non-conforming elements the reconstruction procedure simplifies to an element-based reconstruction as shown in~\cite{BerMolSta:17}.

\section{Conclusion}

In this paper we discussed the equilibrated flux reconstruction by Braess and Sch\"oberl~\cite{BraSch:08,BraPilSch:09}, stemming from the classical Prager--Synge hypercircle theory~\cite{PraSyn:47}. We recalled the a posteriori error analysis for both an elementwise estimator $\eta^\Delta$ and a patchwise estimator $\eta^\patchsymbol{}$, and we showed how to adapt the abstract theory of~\cite{MR2875241} in order to prove the optimal convergence of the adaptive scheme based on those estimators.

\bibliographystyle{amsalpha}
\bibliography{references}

\newcommand{\etalchar}[1]{$^{#1}$}
\providecommand{\bysame}{\leavevmode\hbox to3em{\hrulefill}\thinspace}
\providecommand{\MR}{\relax\ifhmode\unskip\space\fi MR }
\providecommand{\MRhref}[2]{%
  \href{http://www.ams.org/mathscinet-getitem?mr=#1}{#2}
}
\providecommand{\href}[2]{#2}
\begin{thebibliography}{RDPE{\etalchar{+}}17}

\bibitem[AO93]{AinOde:93}
M.~Ainsworth and J.~T. Oden, \emph{A unified approach to a posteriori error
  estimation using element residual methods}, Numer. Math. \textbf{65} (1993),
  23--50.

\bibitem[AO00]{AinOde:00}
Mark Ainsworth and J.~Tinsley Oden, \emph{A posteriori error estimation in
  finite element analysis}, Wiley, New York, 2000.

\bibitem[AR10]{AinRan:10}
M.~Ainsworth and R.~Rankin, \emph{Guaranteed computable error bounds for
  conforming and nonconforming finite element analyses in planar elasticity},
  Int. J. Numer. Meth. Engng. \textbf{82} (2010), 1114--1157.

\bibitem[AW02]{ArnWin:02}
D.~N. Arnold and R.~Winther, \emph{Mixed finite elements for elasticity},
  Numer. Math. \textbf{92} (2002), 401--419.

\bibitem[BBS19]{BBS}
F.~Bertrand, D.~Boffi, and R.~Stenberg, \emph{Asymptotically exact a posteriori
  error analysis for the mixed {L}aplace eigenvalue problem}, Comput. Methods
  Appl. Math. (2019), to appear.

\bibitem[BFH14]{MR3249368}
D.~Braess, T.~Fraunholz, and R.~H.~W. Hoppe, \emph{An equilibrated a posteriori
  error estimator for the interior penalty discontinuous {G}alerkin method},
  SIAM J. Numer. Anal. \textbf{52} (2014), no.~4, 2121--2136. \MR{3249368}

\bibitem[BHS08]{MR2425501}
Dietrich Braess, Ronald H.~W. Hoppe, and Joachim Sch\"{o}berl, \emph{A
  posteriori estimators for obstacle problems by the hypercircle method},
  Comput. Vis. Sci. \textbf{11} (2008), no.~4-6, 351--362. \MR{2425501}

\bibitem[BKMSa]{bertrandCISM}
Fleurianne Bertrand, Bernhard Kober, Marcel Moldenhauer, and Gerhard Starke,
  \emph{Equilibrated stress reconstruction and a posteriori error estimation
  for linear elasticity}.

\bibitem[BKMSb]{BerKobMolSta:19}
\bysame, \emph{Weakly symmetric stress equilibration and a posteriori error
  estimation for linear elasticity}, submitted for publication, arXiv:
  1808.02655.

\bibitem[BMS10]{MR2595045}
Dietrich Braess, Pingbing Ming, and Zhong-Ci Shi, \emph{Shear locking in a
  plane elasticity problem and the enhanced assumed strain method}, SIAM J.
  Numer. Anal. \textbf{47} (2010), no.~6, 4473--4491. \MR{2595045}

\bibitem[BMS18]{BerMolSta:17}
F.~Bertrand, M.~Moldenhauer, and G.~Starke, \emph{A posteriori error estimation
  for planar linear elasticity by stress reconstruction}, Comput. Methods Appl.
  Math. (2018).

\bibitem[BPS09a]{BraPilSch:09}
D.~Braess, V.~Pillwein, and J.~Sch{\"o}berl, \emph{Equilibrated residual error
  estimates are $p$-robust}, Comput. Methods Appl. Mech. Engrg. \textbf{198}
  (2009), 1189--1197.

\bibitem[BPS09b]{MR2500243}
Dietrich Braess, Veronika Pillwein, and Joachim Sch\"{o}berl,
  \emph{Equilibrated residual error estimates are {$p$}-robust}, Comput.
  Methods Appl. Mech. Engrg. \textbf{198} (2009), no.~13-14, 1189--1197.
  \MR{2500243}

\bibitem[Bra09]{MR2520373}
Dietrich Braess, \emph{An a posteriori error estimate and a comparison theorem
  for the nonconforming {$P_1$} element}, Calcolo \textbf{46} (2009), no.~2,
  149--155. \MR{2520373}

\bibitem[Bra13]{Bra:13}
D.~Braess, \emph{Finite {E}lemente: {T}heorie, schnelle {L}{\"o}ser und
  {A}nwendungen in der {E}lastizit\"atstheorie}, Springer, Berlin, 2013, 5.
  Auflage.

\bibitem[BS08a]{BraSch:08}
D.~Braess and J.~Sch\"{o}berl, \emph{Equilibrated residual error estimator for
  edge elements}, Math. Comp. \textbf{77} (2008), no.~262, 651--672.
  \MR{2373174}

\bibitem[BS08b]{MR2373174}
Dietrich Braess and Joachim Sch\"{o}berl, \emph{Equilibrated residual error
  estimator for edge elements}, Math. Comp. \textbf{77} (2008), no.~262,
  651--672. \MR{2373174}

\bibitem[CDM{\etalchar{+}}17]{MR3702871}
Eric Canc\`es, Genevi\`eve Dusson, Yvon Maday, Benjamin Stamm, and Martin
  Vohral\'{i}k, \emph{Guaranteed and robust a posteriori bounds for {L}aplace
  eigenvalues and eigenvectors: conforming approximations}, SIAM J. Numer.
  Anal. \textbf{55} (2017), no.~5, 2228--2254. \MR{3702871}

\bibitem[CFPP14]{CarFeiPagPra:14}
C.~Carstensen, M.~Feischl, M.~Page, and D.~Praetorius, \emph{Axioms of
  adaptivity}, Comput. Math. Appl. \textbf{67} (2014), no.~6, 1195--1253.
  \MR{3170325}

\bibitem[CFPV09]{MR2559737}
Ibrahim Cheddadi, Radek Fu\v{c}\'{i}k, Mariana~I. Prieto, and Martin
  Vohral\'{i}k, \emph{Guaranteed and robust a posteriori error estimates for
  singularly perturbed reaction-diffusion problems}, M2AN Math. Model. Numer.
  Anal. \textbf{43} (2009), no.~5, 867--888. \MR{2559737}

\bibitem[CM13]{MR3018142}
C.~Carstensen and C.~Merdon, \emph{Effective postprocessing for equilibration a
  posteriori error estimators}, Numer. Math. \textbf{123} (2013), no.~3,
  425--459. \MR{3018142}

\bibitem[CN12]{MR2875241}
J.~Manuel Casc\'{o}n and Ricardo~H. Nochetto, \emph{Quasioptimal cardinality of
  {AFEM} driven by nonresidual estimators}, IMA J. Numer. Anal. \textbf{32}
  (2012), no.~1, 1--29. \MR{2875241}

\bibitem[CNT17]{MR3649425}
Emmanuel Creus\'{e}, Serge Nicaise, and Roberta Tittarelli, \emph{A guaranteed
  equilibrated error estimator for the {$\mathbf{A}-\varphi$} and
  {$\mathbf{T}-\Omega$} magnetodynamic harmonic formulations of the {M}axwell
  system}, IMA J. Numer. Anal. \textbf{37} (2017), no.~2, 750--773.
  \MR{3649425}

\bibitem[CR17]{MR3719030}
C.~Carstensen and H.~Rabus, \emph{Axioms of adaptivity with separate marking
  for data resolution}, SIAM J. Numer. Anal. \textbf{55} (2017), no.~6,
  2644--2665. \MR{3719030}

\bibitem[CZ12a]{CaiZha:12b}
Z.~Cai and S.~Zhang, \emph{Mixed methods for stationary {N}avier-{S}tokes
  equations based on pseudostress-pressure-velocity formulation}, Math. Comp.
  \textbf{81} (2012), 1903--1927.

\bibitem[CZ12b]{MR2888308}
Zhiqiang Cai and Shun Zhang, \emph{Robust equilibrated residual error estimator
  for diffusion problems: conforming elements}, SIAM J. Numer. Anal.
  \textbf{50} (2012), no.~1, 151--170. \MR{2888308}

\bibitem[DEV13]{MR3033032}
V\'{i}t Dolej\v{s}\'{i}, Alexandre Ern, and Martin Vohral\'{i}k, \emph{A
  framework for robust a posteriori error control in unsteady nonlinear
  advection-diffusion problems}, SIAM J. Numer. Anal. \textbf{51} (2013),
  no.~2, 773--793. \MR{3033032}

\bibitem[DEV16]{MR3556071}
\bysame, \emph{{$hp$}-adaptation driven by polynomial-degree-robust a
  posteriori error estimates for elliptic problems}, SIAM J. Sci. Comput.
  \textbf{38} (2016), no.~5, A3220--A3246. \MR{3556071}

\bibitem[DM99]{MR1648383}
Philippe Destuynder and Brigitte M\'{e}tivet, \emph{Explicit error bounds in a
  conforming finite element method}, Math. Comp. \textbf{68} (1999), no.~228,
  1379--1396. \MR{1648383}

\bibitem[DPVY15]{MR3266956}
Daniele~A. Di~Pietro, Martin Vohral\'{i}k, and Soleiman Yousef, \emph{Adaptive
  regularization, linearization, and discretization and a posteriori error
  control for the two-phase {S}tefan problem}, Math. Comp. \textbf{84} (2015),
  no.~291, 153--186. \MR{3266956}

\bibitem[ESV10]{MR2601287}
Alexandre Ern, Annette~F. Stephansen, and Martin Vohral\'{i}k, \emph{Guaranteed
  and robust discontinuous {G}alerkin a posteriori error estimates for
  convection-diffusion-reaction problems}, J. Comput. Appl. Math. \textbf{234}
  (2010), no.~1, 114--130. \MR{2601287}

\bibitem[EV15a]{ErnVoh:15}
A.~Ern and M.~Vohral{\'{\i}}k, \emph{Polynomial-degree-robust a posteriori
  error estimates in a unified setting for conforming, nonconforming,
  discontinuous {G}alerkin, and mixed discretizations}, SIAM J. Numer. Anal.
  \textbf{53} (2015), 1058--1081.

\bibitem[EV15b]{MR3335498}
Alexandre Ern and Martin Vohral\'{i}k, \emph{Polynomial-degree-robust a
  posteriori estimates in a unified setting for conforming, nonconforming,
  discontinuous {G}alerkin, and mixed discretizations}, SIAM J. Numer. Anal.
  \textbf{53} (2015), no.~2, 1058--1081. \MR{3335498}

\bibitem[HSV12]{MR2995179}
Antti Hannukainen, Rolf Stenberg, and Martin Vohral\'{i}k, \emph{A unified
  framework for a posteriori error estimation for the {S}tokes problem}, Numer.
  Math. \textbf{122} (2012), no.~4, 725--769. \MR{2995179}

\bibitem[HW12]{MR2872024}
S.~H\"{u}eber and B.~Wohlmuth, \emph{Equilibration techniques for solving
  contact problems with {C}oulomb friction}, Comput. Methods Appl. Mech. Engrg.
  \textbf{205/208} (2012), 29--45. \MR{2872024}

\bibitem[Kim12]{MR2980729}
Kwang-Yeon Kim, \emph{Flux reconstruction for the {$P2$} nonconforming finite
  element method with application to a posteriori error estimation}, Appl.
  Numer. Math. \textbf{62} (2012), no.~12, 1701--1717. \MR{2980729}

\bibitem[KS11]{MR2776915}
Christian Kreuzer and Kunibert~G. Siebert, \emph{Decay rates of adaptive finite
  elements with {D}\"{o}rfler marking}, Numer. Math. \textbf{117} (2011),
  no.~4, 679--716. \MR{2776915}

\bibitem[LL83]{LadLeg:83}
P.~Ladev\`{e}ze and D.~Leguillon, \emph{Error estimate procedure in the finite
  element method and applications}, SIAM J. Numer. Anal. \textbf{20} (1983),
  485--509.

\bibitem[LO13]{MR3061473}
Xuefeng Liu and Shin'ichi Oishi, \emph{Verified eigenvalue evaluation for the
  {L}aplacian over polygonal domains of arbitrary shape}, SIAM J. Numer. Anal.
  \textbf{51} (2013), no.~3, 1634--1654. \MR{3061473}

\bibitem[MN17]{MR3621260}
Zoubida Mghazli and Ilyas Naji, \emph{Analyse {\it a posteriori} d'erreur par
  reconstruction pour un mod\`ele d'\'{e}coulement dans un milieu poreux
  fractur\'{e}}, C. R. Math. Acad. Sci. Paris \textbf{355} (2017), no.~3,
  304--309. \MR{3621260}

\bibitem[NWW08]{NicWitWoh:08}
S.~Nicaise, K.~Witowski, and B.~Wohlmuth, \emph{An a posteriori error estimator
  for the {L}am{\'e} equation based on equilibrated fluxes}, IMA J. Numer.
  Anal. \textbf{28} (2008), 331--353.

\bibitem[PS47]{PraSyn:47}
W.~Prager and J.~L. Synge, \emph{Approximations in elasticity based on the
  concept of function space}, Quart. Appl. Math. \textbf{5} (1947), 241--269.

\bibitem[PVWW13]{MR3033022}
Gergina~V. Pencheva, Martin Vohral\'{i}k, Mary~F. Wheeler, and Tim Wildey,
  \emph{Robust a posteriori error control and adaptivity for multiscale,
  multinumerics, and mortar coupling}, SIAM J. Numer. Anal. \textbf{51} (2013),
  no.~1, 526--554. \MR{3033022}

\bibitem[RDPE{\etalchar{+}}17]{MR3622156}
Rita Riedlbeck, Daniele~A. Di~Pietro, Alexandre Ern, Sylvie Granet, and Kyrylo
  Kazymyrenko, \emph{Stress and flux reconstruction in {B}iot's poro-elasticity
  problem with application to a posteriori error analysis}, Comput. Math. Appl.
  \textbf{73} (2017), no.~7, 1593--1610. \MR{3622156}

\bibitem[RSS04]{MR2114285}
Sergey Repin, Stefan Sauter, and Anton Smolianski, \emph{A posteriori
  estimation of dimension reduction errors for elliptic problems on thin
  domains}, SIAM J. Numer. Anal. \textbf{42} (2004), no.~4, 1435--1451.
  \MR{2114285}

\bibitem[RSS07]{MR2318795}
\bysame, \emph{Two-sided a posteriori error estimates for mixed formulations of
  elliptic problems}, SIAM J. Numer. Anal. \textbf{45} (2007), no.~3, 928--945.
  \MR{2318795}

\bibitem[Ste07]{MR2324418}
R.~Stevenson, \emph{Optimality of a standard adaptive finite element method},
  Found. Comput. Math. \textbf{7} (2007), no.~2, 245--269. \MR{2324418}

\bibitem[Ste08]{MR2353951}
Rob Stevenson, \emph{The completion of locally refined simplicial partitions
  created by bisection}, Math. Comp. \textbf{77} (2008), no.~261, 227--241.
  \MR{2353951}

\bibitem[TW13]{MR3129761}
Simon Tavener and Tim Wildey, \emph{Adjoint based a posteriori analysis of
  multiscale mortar discretizations with multinumerics}, SIAM J. Sci. Comput.
  \textbf{35} (2013), no.~6, A2621--A2642. \MR{3129761}

\bibitem[Ver09]{MR2551163}
R.~Verf\"{u}rth, \emph{A note on constant-free a posteriori error estimates},
  SIAM J. Numer. Anal. \textbf{47} (2009), no.~4, 3180--3194. \MR{2551163}

\bibitem[Voh10]{MR2684353}
Martin Vohral\'{i}k, \emph{Unified primal formulation-based a priori and a
  posteriori error analysis of mixed finite element methods}, Math. Comp.
  \textbf{79} (2010), no.~272, 2001--2032. \MR{2684353}

\bibitem[Voh11]{MR2765501}
\bysame, \emph{Guaranteed and fully robust a posteriori error estimates for
  conforming discretizations of diffusion problems with discontinuous
  coefficients}, J. Sci. Comput. \textbf{46} (2011), no.~3, 397--438.
  \MR{2765501}

\bibitem[VY18]{MR3761018}
Martin Vohral\'{i}k and Soleiman Yousef, \emph{A simple a posteriori estimate
  on general polytopal meshes with applications to complex porous media flows},
  Comput. Methods Appl. Mech. Engrg. \textbf{331} (2018), 728--760.
  \MR{3761018}

\bibitem[WW10]{MR2684731}
Alexander Weiss and Barbara~I. Wohlmuth, \emph{A posteriori error estimator for
  obstacle problems}, SIAM J. Sci. Comput. \textbf{32} (2010), no.~5,
  2627--2658. \MR{2684731}

\bibitem[Zha06]{MR2241823}
Sheng Zhang, \emph{On the accuracy of {R}eissner-{M}indlin plate model for
  stress boundary conditions}, M2AN Math. Model. Numer. Anal. \textbf{40}
  (2006), no.~2, 269--294. \MR{2241823}

\end{thebibliography}

\end{document}